\documentclass[10pt,leqno]{amsart}
\usepackage{graphicx}
\usepackage{indentfirst,csquotes}

\usepackage{amsfonts}
\usepackage{amsmath} 
\usepackage{amssymb}
\usepackage{array} 
\usepackage{geometry}

\usepackage{xcolor}

\usepackage{blindtext}

\usepackage{amssymb,latexsym}
\theoremstyle{plain}
\newtheorem{theorem}{Theorem}
\newtheorem{corollary}{Corollary}

\newtheorem{lemma}{Lemma}
\newtheorem{proposition}{Proposition}
\theoremstyle{definition}

\theoremstyle{remark}

\numberwithin{equation}{section}

\newcommand{\psum}{\sideset{}{^*}\sum}
\newcommand{\dsum}{\sideset{}{^d}\sum}

\author[]{}
\date{\today}
\address{}
\email{}

\title{Averaging quadratically twisted modular $L$-values and their derivatives}
\author{Tinghao Huang}
\date{June 2025}

\begin{document}

\let\thefootnote\relax

\begin{abstract}
In this paper, we unconditionally establish an asymptotic formula for the product of the quadratically twisted central $L$-value associated to a holomorphic cusp form $f$, and the quadratically twisted central $L$-derivative to a distinct holomorphic cusp form $g$. This result may be viewed as an extension of \cite{Li-MR4768632}, \cite{Kumar.etc-MR4765788} and \cite{zhou2025momentderivativesquadratictwists}.
\end{abstract} 
\maketitle

\maketitle

\section{Introduction}
Let $f(z)$ and $g(z)$ be two cuspidal Hecke newforms of even weights $\kappa_1, \kappa_2$ and odd levels $q_1, q_2$ respectively.
The Fourier expansions of $f$ and $g$ at the cusp $\infty$ shall take the form
\[
    f(z) = \sum_{m=1}^{\infty}\lambda_f(m)m^{\frac{\kappa_1 - 1}{2}}e^{2\pi imz},\,  g(z) = \sum_{n=1}^{\infty}\lambda_g(n)n^{\frac{\kappa_2 - 1}{2}}e^{2\pi inz}.
\] We normalize $f$ and $g$ so that $\lambda_f(1) = \lambda_g(1) = 1$.
For $\chi_D(\cdot) = \big(\frac{D}{\cdot}\big)$ a primitive real character where $D$ is a fundamental discriminant (that is, discriminant for a quadratic extension of $\mathbb{Q}$), we may construct the twisted $L$-functions associated to $f$ and $g$ as the following absolutely convergent Dirichlet series whenever the real part of $s$ is greater than 1:
\[
    L(s,f\otimes \chi_D) := \sum^{\infty}_{m=1}\lambda_f(m)\chi_D(m)m^{-s}
\]
and
\[
    L(s,g\otimes \chi_D) := \sum^{\infty}_{n=1}\lambda_g(n)\chi_D(n)n^{-s}.
\]
These Dirichlet series could be modified to the completed $L$-functions $\Lambda(f,f \otimes \chi_D)$ and $\Lambda(s,g \otimes \chi_D)$ by setting 
\[
    \Lambda(s,f\otimes \chi_D) = 
    \bigg(\frac{|D|\sqrt{q_1}}{2\pi}\bigg)^s\Gamma\bigg(s + \frac{\kappa_1-1}{2}\bigg)\cdot L(s,f\otimes \chi_D)
\]
and 
\[
    \Lambda(s,g\otimes \chi_D) = 
    \bigg(\frac{|D|\sqrt{q_2}}{2\pi}\bigg)^s\Gamma\bigg(s + \frac{\kappa_2-1}{2}\bigg)\cdot L(s,g\otimes \chi_D),
\]
then extended to be entire functions defined over $\mathbb{C}$, and shown to satisfy functional equations
\[
    \Lambda(s,f\otimes \chi_D) 
    = i^{\kappa_1}\eta_f\cdot \chi_D(-q_1)\Lambda(1-s, f\otimes \chi_D)
\]
and 
\[
    \Lambda(s,g\otimes \chi_D) 
    = i^{\kappa_2}\eta_g\cdot \chi_D(-q_2)\Lambda(1-s, g\otimes \chi_D).
\] Here $\eta_f$ and $\eta_g$ are the eigenvalues of the Fricke involutions of $f$ and $g$, which are independent to $D$ and always take values $\pm 1$.
For notational convenience, we denote $i^{\kappa_1}\eta_f\cdot \chi_D(-q_1)$ and $i^{\kappa_2}\eta_g\cdot \chi_D(-q_2)$ respectively as $\omega(f\otimes \chi_D)$ and $\omega(g\otimes \chi_D)$. Notice that with weights $\kappa_1, \kappa_2$ even, $\omega(f\otimes \chi_D)$ and $\omega(g\otimes \chi_D)$ can only take values $\pm 1$.

In this paper, we aim to calculate unconditionally an asymptotic formula for the following mixed moment
\begin{equation}\label{momentequation}
    S_J(f,g';X):=\psum_{\substack{(d,2q_1q_2) = 1 \\ \omega(f\otimes \chi_{8d}) = 1 \\ \omega(g\otimes \chi_{8d}) = -1}}L(1/2, f\otimes \chi_{8d})\cdot L'(1/2,g\otimes \chi_{8d})\cdot J\bigg(\frac{8d}{X}\bigg)
\end{equation}
as $X \rightarrow \infty$. Note that for odd $d$, $8d$ is a fundamental discriminant. In the above, the asterisk on the top of the summation sign simply means that the summation is over square free integers, and $J(x)$ is taken to be a non-negative smooth function supported on the interval $[1/2,2]$.
Our main theorem states the following:
\begin{theorem}
    With the notation above, we have 
    \[
        S_J(f,g';X) = C_{f,g}\check{J}(0)X\log X + O_{f,g,J}(X(\log \log X)^5),
    \]
    in which $C_{f,g}$ is a constant depending only on $f$ and $g$, $\check{J}(0) = \int^{\infty}_{-\infty}J(x)dx$, and the implied constant in the error term depends only on the forms $f,g$ and $J(\cdot)$. Moreover, $C_{f,g} = 0$ if and only if $i^{\kappa_1}\eta_f = -1$ and $q_1$ is a square or $i^{\kappa_2}\eta_g = 1$ and $q_2$ is a square, in which case the moment $S_J(f,g';X) = 0$ identically.
    
\end{theorem}

The study of moments of quadratically twisted central $L$-values and their derivatives associated with cusp forms has a rich history and deep applications in arithmetic geometry and the theory of automorphic representations. For foundational results on the first moments of twisted central $L$-values and their derivatives, we refer, for example, to \cite{BumpFriedbergHoffsteinMR1038358}, \cite{MurtyMurtyMR1109350}, \cite{Iwaniec-MR1081731}, \cite{Petrow-MR3180602}, and \cite{LuoRamakrishnanMR1474162}.

An asymptotic formula for the more challenging second moments of quadratically twisted central $L$-values of cusp forms was first established by Soundararajan and Young \cite{YoungSoundMR2677611}, where they proved an asymptotic formula of the form
\[
    \sum_{\substack{d \leq X\\ (d,2)=1}} L(1/2,f\otimes \chi_{8d})^2 = (c + o(1))X\log X
\]
under the assumption of the Generalized Riemann Hypothesis (GRH) for certain $L$-functions. Here, $c$ is a constant depending only on $f$. This result was later proven unconditionally by Li in \cite{Li-MR4768632}, thereby removing the dependence on GRH.

More recently, an asymptotic formula for the second moment of twisted central derivatives $L'(1/2, f\otimes \chi_{8d})$ was obtained by Kumar, Mallesham, Sharma, and Singh in \cite{Kumar.etc-MR4765788}, building on crucial estimates of bilinear forms from \cite{Li-MR4768632}. In a very recent work, Zhou \cite{zhou2025momentderivativesquadratictwists} established an asymptotic formula for the more delicate (smoothed) first moment of products of twisted central derivatives:
\[
    \sum_{d} L'(1/2,f\otimes \chi_{8d}) \cdot L'(1/2,g\otimes \chi_{8d}) F(8d/X),
\]
where $F \in C^{\infty}_c(\mathbb{R}^+)$ and $f$, $g$ are two \textit{distinct} cusp forms. However, with current techniques, the conjectured asymptotic formula for the product of twisted central values of two distinct forms $f$ and $g$,
\[
    \sum_d L(1/2,f\otimes \chi_{8d}) L(1/2,g\otimes \chi_{8d}) F(8d/X),
\]
remains out of reach.

Our main theorem also has implications for properties of elliptic curves over $\mathbb{Q}$. It is well known that for any elliptic curve $E$ over $\mathbb{Q}$ with conductor $q$ and Weierstrass equation of the form
\[
    E: y^2 = x^3 + Ax + B,\quad A,B\in \mathbb{Q},
\]
there exists (not necessarily uniquely) a cuspidal Hecke eigen-newform $f_E$ of level $q$ and weight $2$ such that the associated Hecke $L$-function agrees with the Hasse--Weil $L$-function of $E$, i.e., $L(s,f_E) = L(s,E)$. 

Given a fundamental discriminant $D$, one defines the quadratic twist of $E$, denoted $E^{(D)}$, by modifying the Weierstrass coefficients:
\[
    E^{(D)}: y^2 = x^3 + AD^2x + BD^3,
\]
which is again an elliptic curve over $\mathbb{Q}$. The $L$-function of this twist satisfies $L(s,E^{(D)}) = L(s,f_E\otimes \chi_D)$. After normalizing these $L$-functions so that the center of the critical strip is $s = \frac{1}{2}$, the celebrated Birch and Swinnerton-Dyer conjecture predicts that the order of vanishing of $L(s,E)$ (and thus $L(s,f_E)$), as well as $L(s,E^{(D)})$ (and $L(s,f_E \otimes \chi_D)$), at $s = \frac{1}{2}$ equals the rank of the corresponding Mordell--Weil group.

By the results of Kolyvagin \cite{KolyvaginMR954295} and Gross--Zagier \cite{GrossZagierMR833192} (see also Section C.16 of \cite{SilvermanMR2514094}), this conjecture is known to hold when the order of vanishing is at most $1$. As a consequence, our Theorem~1 yields the following corollary:

\begin{corollary}
    Let $E_1$ and $E_2$ be elliptic curves over $\mathbb{Q}$ with conductors $q_1$ and $q_2$, respectively, neither of which is a perfect square. Then there exist infinitely many fundamental discriminants of the form $8d$, where $(d,2)=1$, such that the Mordell--Weil rank of $E_1^{(8d)}$ is equal to $1$, and that of $E_2^{(8d)}$ is equal to $0$.
\end{corollary}

\section{Notations}
In this section, we briefly clarify the notations that will be frequently used throughout the paper. We adopt the Vinogradov symbols $\ll$ and $O(\cdot)$, with the understanding that the dependence of the implied constants will be specified when necessary. Additionally, $\Re(\cdot)$ will denote the real part of a complex variable.

As previously mentioned, we use $\psum$ to represent a summation over square-free integers. The notation $\dsum_N$ will be used to denote a summation over dyadic intervals, where $N$ takes values of the form $2^j$ with $j \geq 0$.

\section{Basic Tools}
In this section we record some tools that shall be crucial to our calculation, as well as some notations that will be convenient for our use in the following sections.

When estimating summations over fundamental discriminants, we will for several times apply the summation formula of Poisson type for real characters and extract the main contribution from the term corresponding to the zero frequency. We record the above summation formula, which was originally due to Soundararajan \cite[Lemma 2.6]{SoundMR1804529}, as the following lemma (see also \cite[Lemma 2.3]{Li-MR4768632}): 
\begin{lemma}\label{Poisson}
    For $F\in C^{\infty}_c(\mathbb{R}^+)$, $n$ odd and $Z > 0$, we have 
    \[
        \sum_{(d,2)=1}\bigg(\frac{d}{n}\bigg)F\bigg(\frac{d}{Z}\bigg) = \frac{Z}{2n}\bigg(\frac{2}{n}\bigg)\sum_{k \in \mathbb{Z}}(-1)^kG_k(n)\check{F}\bigg(\frac{kZ}{2n}\bigg),
    \]
    in which 
    \[
        G_k(n) := \bigg(\frac{1-i}{2} + \bigg(\frac{-1}{n}\bigg)\frac{1+i}{2}\bigg)\sum_{a\,( \text{mod}\,n)}\bigg(\frac{a}{n}\bigg)e\bigg(\frac{ak}{n}\bigg)
    \] is the Gauss-type sum, and 
    \[
        \check{F}(\xi ) = \int_{\mathbb{R}}(\cos(2\pi x\xi) + \sin(2\pi x\xi))F(x)dx
    \]
    is an integral transform of $F$.  
\end{lemma}
For $F\in C^{\infty}_c(\mathbb{R}^+)$, we may invoke Mellin inversion to rewrite $\check{F}$ as 
\[
    \check{F}(\xi) = 
    \frac{1}{2\pi i}\int_{(1/2)}\tilde{F}(1-s)\Gamma(s)(\cos + \text{sgn}(\xi)\sin)(\pi s/2)(2\pi|\xi|)^{-s}ds,
\]
where $\tilde{F}(s) := \int^{\infty}_0F(x)x^{s-1}dx$ is the Mellin transform of $F$. We also note that with $\Re(s)$ bounded, $\tilde{F}(s)$ acquires the following rapid decay property:
\[
    \tilde{F}(s) \ll_A (1+|s|)^{-A}
\]
for any $A>0$.

It shall be useful to explicitly evaluate the Gauss-type sums $G_k(\cdot)$'s appeared on the right hand side of Poisson summation. This work was accomplished in \cite[Lemma 2.3]{SoundMR1804529}, which shall be quoted as below:
\begin{lemma}\label{Gauss sums}
    For odd $m,n$ co-prime, we have 
    \[
        G_k(mn) = G_k(m)G_k(n).
    \]
    Also, if $p^{\alpha} || k$ (setting $\alpha = \infty$ if $k=0$), then 
    \[G_k(p^{\beta}) = \left\{
    \begin{aligned}
    0, & & \beta \leq \alpha \,\text{is odd}, \\
    \phi(p^{\beta}), & & \beta \leq \alpha\,\text{is even}, \\
    -p^{\alpha}, & & \beta = \alpha + 1 \,\text{is even}, \\ 
    \bigg(\frac{kp^{-\alpha}}{p}\bigg)p^{\alpha+1/2}, & & \beta = \alpha + 1\,\text{is odd}, \\
    0, & & \beta \geq \alpha + 2. 
    \end{aligned} \right.
\]
\end{lemma}

Oftenly, it would be convenient to dyadically partition summations, and we introduce a family of functions to aid us do so. We introduce $G \in C^{\infty}_c([3/4,2])$ satisfying
\begin{equation*}
    \begin{aligned}
        & G(x) = 1, \,x \in [1,3/2] \\
        & G(x) + G(x/2) = 1, \, x \in [1,3]. 
    \end{aligned}
\end{equation*}
Such function could be explicitly constructed (see \cite[Lemma 1.10]{WarnerMR722297}), and would satisfy 
\[
    G(x) + ... + G(x/2^J) = 1
\]
for any $x \in [1,3\cdot 2^{J-1}]$, in which $J\geq 1$ and $G(x) + ... + G(x/2^J)$ will be supported on $[3/4,2^{J+1}]$. We shall fix a choice of $G(x)$ throughout the rest of the paper. Clearly $\dsum_N G(x/N)$ is a well-defined function supported on $[3/4,\infty)$, and is equal to 1 for any $x \geq 1$.

Finally, the following bilinear form estimation due to Li \cite[Lemma 6.3]{Li-MR4768632} (with minor technical modification to generalize to arbitrary levels) shall be crucial for our use in section 7:
\begin{lemma}\label{Li-biliear}
    For any real $\mathcal{X}, N \geq 1$ and real $t$ as well as a positive integer $q$, we have 
    \[
        \sum_{\substack{(d,2)=1\\ d\leq \mathcal{X}}}\bigg|
        \sum_{(n,q)=1}\frac{\lambda_f(n)}{n^{1/2+it}}\bigg(\frac{8d}{n}\bigg)G(n/N)
        \bigg|^2 \ll 
        d(q)^5\mathcal{X}(1+|t|)^3\log(2+|t|),
    \]
where $d(\cdot)$ is the usual divisor counting function. Here the implied constant only depends on $f$.
\end{lemma}

\section{Set-up \& statement of main propositions}
A standard technique to express the central $L$-values or derivatives in terms of Fourier coefficients $\lambda_f(m), \lambda_g(n)$ will be the application of
approximate functional equations. We record the approximate functional equations for $L(1/2,f\otimes \chi_{8d})$ and $L'(1/2,g\otimes \chi_{8d})$ as below.
\begin{lemma}\label{ApprFuncEq}
    With the notations as above, we have 
    \begin{equation}\label{L(f)-AFE}
        L(1/2, f\otimes \chi_{8d}) = 
        \bigg(1+\omega(f\otimes \chi_{8d})\bigg)\cdot \sum_{m \geq 1}\frac{\lambda_f(m)\chi_{8d}(m)}{\sqrt{m}}W_1\bigg(\frac{m}{8|d|\sqrt{q_1}}\bigg) 
    \end{equation}
    and 
    \begin{equation}\label{L'(g)-AFE}
         L'(1/2, g\otimes \chi_{8d}) = 
         \bigg(1- \omega(g\otimes \chi_{8d})\bigg)\cdot \sum_{n \geq 1}\frac{\lambda_g(n)\chi_{8d}(n)}{\sqrt{n}}W_2\bigg(\frac{n}{8|d|\sqrt{q_2}}\bigg),
    \end{equation}
    in which
    \[
        W_1(x) = \frac{1}{2\pi i}\int_{(3)}(2\pi x)^{-w}\cdot \frac{\Gamma(\kappa_1/2+w)}{\Gamma(\kappa_1/2)}\frac{dw}{w}
    \]
    while
    \[
        W_2(y) = \frac{1}{2\pi i}\int_{(3)}(2\pi y)^{-u}\cdot \frac{\Gamma(\kappa_2/2 + u)}{\Gamma(\kappa_2/2)}\frac{du}{u^2}.
    \]
\end{lemma}
\begin{proof}
    These are consequences of the holomorphy as well as the functional equations of the complete $L$-functions $\Lambda(s,f\otimes \chi_{8d})$ and $\Lambda(s,g\otimes \chi_{8d})$, after an application of the Residue Theorem. See \cite[Theorem 5.3]{IK-MR2061214} and \cite[Lemma 3.1]{Petrow-MR3180602} for further details.
\end{proof}

For notational simplicity, we shall denote $\frac{\Gamma_1(\kappa_1/2+w)}{\Gamma(\kappa_1/2)}$ and $\frac{\Gamma_2(\kappa_2/2+u)}{\Gamma(\kappa_2/2)}$ respectively by $\gamma_1(w)$ and $\gamma_2(u)$.
Notice that by Stirling's formula, we have the following estimation for $\gamma_1,\gamma_2$ when $\Re(w)$ and $\Re(u)$ are bounded:
\[
    \gamma_1(w) \ll_{\kappa_1} (1+|w|)^{-20},\,\gamma_2(u) \ll_{\kappa_2} (1 + |u|)^{-20},
\]
in which the implied constants depends only on $\kappa_1$ or $\kappa_2$.

In practice, it would be convenient to "truncate" the right hand sides of \eqref{L(f)-AFE} and \eqref{L'(g)-AFE} by an auxiliary parameter $Y := \frac{X}{(\log X)^{200}}$, and set
\[
    \mathcal{A}_f = \mathcal{A}_f(Y; 8d) := \bigg(1+\omega(f\otimes \chi_{8d})\bigg)\cdot \sum_{m \geq 1}\frac{\lambda_f(m)\chi_{8d}(m)}{\sqrt{m}}W_1\bigg(\frac{m}{Y}\bigg)
\]
while 
\[
    \mathcal{B}_f = \mathcal{B}_f(Y; 8d) := L(1/2,f\otimes \chi_{8d}) - \mathcal{A}_f,
\]
and similarly
\[
    \mathcal{A}'_{g} = \mathcal{A}'_{g}(Y;8d) 
    := \bigg(1-\omega(g\otimes \chi_{8d})\bigg)\cdot \sum_{n \geq 1}\frac{\lambda_g(n)\chi_{8d}(n)}{\sqrt{n}}W_2\bigg(\frac{n}{Y}\bigg)   
\]
and 
\[
    \mathcal{B}'_g = \mathcal{B}'_g(Y; 8d) := L'(1/2,g\otimes \chi_{8d}) - \mathcal{A}'_g.
\]

After truncation, we may rewrite our target sum \eqref{momentequation} as:
\[
    S_J(f,g';X) = \psum_{\substack{(d,2q_1q_2) = 1 \\ \omega(f\otimes \chi_{8d}) = 1 \\ \omega(g\otimes \chi_{8d}) = -1}}L(1/2, f\otimes \chi_{8d})\cdot L'(1/2,g\otimes \chi_{8d})\cdot J\bigg(\frac{8d}{X}\bigg)
    =
    \frac{1}{4}\cdot \big(\mathfrak{I}_1  +\mathfrak{I}_2 +\mathfrak{I}_3 +\mathfrak{I}_4\big), 
\]
in which
\begin{equation*}\label{AfA'g}
    \mathfrak{I}_1 = \psum_{\substack{(d,2q_1q_2) = 1 \\ \omega(f\otimes \chi_{8d}) = 1 \\ \omega(g\otimes \chi_{8d}) = -1}}\mathcal{A}_f\mathcal{A}'_g J\bigg(\frac{8d}{X}\bigg),
\end{equation*}
\begin{equation*}\label{AfB'g}
    \mathfrak{I}_2 = \psum_{\substack{(d,2q_1q_2) = 1 \\ \omega(f\otimes \chi_{8d}) = 1 \\ \omega(g\otimes \chi_{8d}) = -1}}\mathcal{A}_f\mathcal{B}'_g J\bigg(\frac{8d}{X}\bigg),
\end{equation*}
\begin{equation*}\label{BfA'g}
    \mathfrak{I}_3 = \psum_{\substack{(d,2q_1q_2) = 1 \\ \omega(f\otimes \chi_{8d}) = 1 \\ \omega(g\otimes \chi_{8d}) = -1}}\mathcal{B}_f\mathcal{A}'_g J\bigg(\frac{8d}{X}\bigg),
\end{equation*}
and 
\begin{equation*}\label{BfB'g}
    \mathfrak{I}_4 = \psum_{\substack{(d,2q_1q_2) = 1 \\ \omega(f\otimes \chi_{8d}) = 1 \\ \omega(g\otimes \chi_{8d}) = -1}}\mathcal{B}_f\mathcal{B}'_g J\bigg(\frac{8d}{X}\bigg).
\end{equation*}

We summarize the estimation for the above $\mathfrak{I}_1$, $\mathfrak{I}_2$, $\mathfrak{I}_3$ and $\mathfrak{I}_4$ as the following four propositions:
\begin{proposition}\label{I_1-Est}
    With the above notations, we have the following:
    \[
        \mathfrak{I}_1 = C_{f,g}\check{J}(0)X\log X + O_{f,g,\epsilon}(X),
    \]
    where $C_f,g$ is a constant defined as
    \begin{align*}
        C_{f,g} &:=    \frac{L(1,f\otimes g)L(1,\text{Sym}^2f)L(1,\text{Sym}^2 g)}{2\pi ^2}  \\
        &\times \big\{\mathcal{Z}(0,0;1) + i^{\kappa_1}\eta_f\mathcal{Z}(0,0;q_1) - 
        i^{\kappa_2}\eta_g\mathcal{Z}(0,0;q_2)- i^{\kappa_1 + \kappa_2}\eta_f\eta_g\mathcal{Z}(0,0;q_1q_2)\big\}
    \end{align*}
    and the implied constant in $O_{f,g}(\cdot)$ is only dependent on the forms $f,g$ and $\epsilon > 0$. Here 
    $\mathcal{Z}(w, u;Q')$ is an absolutely convergent and uniformly bounded Euler product defined on $\Re(w), \Re(u) > -\frac{1}{4}$. Moreover, $C_{f,g} = 0$ if and only if $i^{\kappa_1}\eta_f = -1$ and $q_1$ is a square or $i^{\kappa_2}\eta_g = 1$ and $q_2$ is a square, in which case the moment $S_J(f,g';X)$ vanishes identically.
\end{proposition}
\begin{proposition}
    With the above notations, we have the following: 
    \[
        \mathfrak{I}_2 \ll_{f,g} X(\log \log X),
    \]
    where the implied constant depends only on the forms $f,g$.
\end{proposition}
\begin{proposition}
    With the notations above, we have the following:
    \[
        \mathfrak{I}_3 \ll_{f,g} X,
    \]
    where the implied constant depends only on the forms $f,g$.
\end{proposition}
and finally 
\begin{proposition}\label{I_4-Est}
    With the above notations, we have the following:
    \[
        \mathfrak{I}_4 \ll X(\log \log X)^{5} ,
    \]
    where the implied constant depends only on the forms $f, g$.
\end{proposition}

Clearly, Theorem 1 follows immediately from Proposition 1 - 4, whose proof shall span the next four sections.

\section{Estimating $\mathfrak{I}_1$}
In this section we treat $\mathfrak{I}_1$ and prove Proposition 1.

Following the technique in \cite{Iwaniec-MR1081731} and \cite[Proposition 6.2]{Li-MR4768632}, we shall prove Proposition \ref{I_1-Est} by applying Poisson summation formula for real characters, which allows us to extract the main term from the zero frequency term and control the magnitude of other terms.  

Upon using the following notation
\[
    S^1_{Q'} := \psum_{\substack{(d,2q_1q_2) = 1}}\sum_{m}\sum_{n}\lambda_f(m)\lambda_g(n)(mn)^{-\frac{1}{2}}\chi_{8d}(mnQ')W_1\bigg(\frac{m}{Y}\bigg)W_2\bigg(\frac{n}{Y}\bigg)J\bigg(\frac{8d}{X}\bigg),
\]
we see that 
\[
    \mathfrak{I}_1 = S^1_1 + i^{\kappa_1}\eta_fS^1_{q_1} - i^{\kappa_2}\eta_gS^1_{q_2} - i^{\kappa_1 + \kappa_2}\eta_f\eta_gS^1_{q_1q_2},
\]
and thus we essentially only need to treat $S^1_{Q'}$ for $Q' = 1,q_1,q_2$ and $q_1q_2$.

By inserting Mobius functions appropriately we may drop the co-prime and square-free conditions in the outermost $d$-summation, and get 
\[  
    S^1_{Q'} = S^{1}_{Q'}(a \leq Z) + S^{1}_{Q'}(a > Z),
\]
in which 
\begin{equation}\label{T(aleqZ)}
\begin{aligned}
    S_{Q'}^{1}(a \leq Z) & = \sum_{\substack {(a,2Q) = 1 \\ a \leq Z}}\mu(a) \sum_{b|Q}\mu(b)\sum_{(m,2a)=1}\sum_{(n,2a)=1}\frac{\lambda_f(m)\lambda_g(n)}{\sqrt{mn}} \\
    & \times \sum_{(d,2)=1}\chi_{8db}(mnQ')W_1\bigg(\frac{m}{Y}\bigg)W_2\bigg(\frac{n}{Y}\bigg)J\bigg(\frac{8a^2bd}{X}\bigg)
\end{aligned}
\end{equation}
and 
\begin{equation*}
\begin{aligned}
    S_{Q'}^{1}(a > Z) & = \sum_{\substack {(a,2Q) = 1 \\ a > Z}}\mu(a)\sum_{(d,2Q)=1}J\bigg(\frac{8a^2d}{X}\bigg)\sum_{(m,2a)=1}\sum_{(n,2a)=1}\frac{\lambda_f(m)\lambda_g(n)}{\sqrt{mn}} \\
    & \times \chi_{8d}(mnQ')W_1\bigg(\frac{m}{Y}\bigg)W_2\bigg(\frac{n}{Y}\bigg)
\end{aligned}
\end{equation*}
in which $Z$ is a large parameter to be determined and $Q = q_1q_2$ for convenience.

To estimate $S_{Q'}^{1}(a>Z)$, we combine Proposition 6.2 of \cite{Li-MR4768632}  (which could be generalized to level $q_1$ with minor technical efforts) and Proposition 3.2 of \cite{Kumar.etc-MR4765788} as well as Cauchy-Schwarz to conclude that 
\begin{equation}\label{S1Q'a>Z}
    S^1_{Q'}(a > Z) \ll \frac{X(\log X)^{2+\epsilon}}{Z^{1-\epsilon}}.
\end{equation}

The estimation of $S_{Q'}^{1}(a \leq Z)$ requires more effort, and we shall see that the main term of the desired asymptotic formula shall be derived from $S_{Q'}^{1}(a \leq Z)$. An application of the Poisson summation formula for real characters Lemma \ref{Poisson} shall transform $S_{Q'}^{1}(a \leq Z)$ as:
\begin{equation*}
    \begin{aligned}
        S^{1}_{Q'}(a \leq Z) & =
        \frac{X\check{J}(0)}{2\pi^2}\cdot \mathop{\sum\sum}_{\substack{mnQ' = \square \\ (mn,2)=1}}\lambda_f(m)\lambda_g(n)(mn)^{-\frac{1}{2}}\\
        & \times \prod_{p|mnQ}(1-p^{-1})\big(\prod_{p|mnQ}(1-p^{-2})^{-1}+O(Z^{-1})\big)\cdot 
        W_1(m/Y)W_2(n/Y) \\
        & + \frac{X}{16}\sum_{\substack{(a,2Q)=1 \\ a \leq Z}}\mu(a)\sum_{b|Q}\mu(b)\sum_{k\neq 0}(-1)^k \\
        &\times \mathop{\sum\sum}_{\substack{mnQ' = \square \\ (mn,2a)=1}}\frac{\lambda_f(m)\lambda_g(n)}{\sqrt{mn}}\frac{\chi_b(mnQ')}{a^2bmnQ'}G_k(mnQ')\check{J}\bigg(\frac{kX}{16mnQ'a^2b}\bigg)W_1(m/Y)W_2(n/Y) \\
        & := S^{1,d}_{Q'}(a \leq Z) + S^{1,n}_{Q'}(a \leq Z) 
    \end{aligned}
\end{equation*}
where $S^{1,d}_{Q'}(a \leq Z)$ and $S^{1,n}_{Q'}(a \leq Z)$ respectively denotes the first and the second term above, and $G_k(\cdot)$ is the quadratic Gauss sum as in Lemma \ref{Gauss sums} and we have used the fact that $G_0(mnQ') = 0$ when $mnQ'$ is not a square and is $\phi(mnQ')$ otherwise. In the first term $S^{1,d}_{Q'}(a \leq Z)$ above we have used the following elementary estimation
\begin{equation}\label{asymptotics for Mobius}
    \sum_{\substack{(a,2mnQ)=1 \\ a \leq Z}}\mu(a)a^{-2} = \frac{8}{\pi^2}\prod_{p|mnQ}(1-p^{-2})^{-1} + O(Z^{-1})
\end{equation}
where the implied constant is absolute. Invoking the multiplicativity of Hecke eigenvalues while unfolding the definitions of $W_1$ and $W_2$, we may transform the "diagonal term" $S^{1,d}_{Q'}(a \leq Z)$ as 
\begin{equation*}
    \begin{aligned}
         S^{1,d}_{Q'}(a \leq Z) & = \frac{X\check{J}(0)}{2\pi^2}\bigg\{\frac{1}{(2\pi i)^2}\int_{(3)}\int_{(3)}\gamma_1(w)\gamma_2(u)\bigg(\frac{Y}{2\pi}\bigg)^{u+w}\cdot w^{-1}u^{-2} \mathcal{E}(w,u;Q')dwdu\\
        & + \cdot 
\mathop{\sum\sum}_{\substack{mnQ' = \square \\ (mn,2)=1}}O(Z^{-1})\cdot \lambda_f(m)\lambda_g(n)(mn)^{-\frac{1}{2}}W_1(m/Y)W_2(n/Y)\bigg\},
    \end{aligned}
\end{equation*}
in which $\mathcal{E}(w,u;Q')$ is the Euler product defined as
\begin{equation*}
    \begin{aligned}
        & \mathcal{E}(w,u;Q') =
         \prod_{p\nmid 2Q}\bigg\{
            1 + \frac{p}{p+1}\bigg(\frac{1}{2}\big(1-\lambda_f(p)p^{-(1/2+w)} + p^{-(1+2w)}\big)^{-1}\big(1-\lambda_g(p)p^{-(1/2+u)} + p^{-(1+2u)}\big)^{-1}  \\
        &    + \frac{1}{2}\big(1+\lambda_f(p)p^{-(1/2+w)} + p^{-(1+2w)}\big)^{-1}\big(1+\lambda_g(p)p^{-(1/2+u)} + p^{-(1+2u)}\big)^{-1}-1
        \bigg)\bigg\} \\
         &\times \prod_{p|Q}\frac{p}{p+1}\bigg\{
        \frac{1}{2}\big(1-\lambda_f(p)p^{-(1/2+w)} + \chi_{0,q_1}(p)p^{-(1+2w)}\big)^{-1}\big(1-\lambda_g(p)p^{-(1/2+u)} + \chi_{0,q_2}(p)p^{-(1+2u)}\big)^{-1}  \\
        &    + (-1)^{\nu_p(Q')}\frac{1}{2}\big(1+\lambda_f(p)p^{-(1/2+w)} + \chi_{0,q_1}(p)p^{-(1+2w)}\big)^{-1}\big(1+\lambda_g(p)p^{-(1/2+u)} + \chi_{0,q_2}(p)p^{-(1+2u)}\big)^{-1}  
        \bigg\},
    \end{aligned}
\end{equation*}
where
$\chi_{0,\mathfrak{q}}(\cdot)$ is the principal character modulo $\mathfrak{q} \in \mathbb{N}$ and $Q' = \prod_{p}p^{\nu_p(Q')}$.
As in \cite{Li-MR4768632} and \cite{zhou2025momentderivativesquadratictwists}, we could relate $\mathcal{E}(w,u;Q')$ to certain $L$-functions associated to the forms $f$ and $g$ as 
\begin{equation*}
    \mathcal{E}(w,u;Q') = L(1+u+w,f\otimes g)
    L(1+2w, \text{Sym}^2f)L(1+2u,\text{Sym}^2g)\mathcal{Z}(w,u;Q')
\end{equation*}
in which $\mathcal{Z}(w,u;Q')$ is an Euler product which is absolutely convergent and uniformly bounded, when $\Re(w),\Re(u) \geq -\frac{1}{4} + \epsilon$ where $\epsilon > 0$ could be taken arbitrarily small.

We directly estimate the second term in the above as
\begin{equation}\label{Z^-1-Estimating}
\begin{aligned}
    & \mathop{\sum\sum}_{\substack{mnQ' = \square \\(mn,2)=1}} O(Z^{-1})\cdot \lambda_f(m)\lambda_g(n)(mn)^{-\frac{1}{2}}W_1(m/Y)W_2(n/Y) \\
    & \ll Z^{-1}\mathop{\sum\sum}_{mnQ' = \square}d(m)d(n)(mn)^{-\frac{1}{2}}|W_1(m/Y)W_2(n/Y)| \\
    & \ll (\log X)^{63}\cdot Z^{-1},
\end{aligned}
\end{equation}
using Deligne's bound and the decay properties of $W_1$ and $W_2$. To bound the second line of the above we have used (1.80) of \cite{IK-MR2061214} and Cauchy-Schwarz.

Now, combining \eqref{Z^-1-Estimating} we shift the two integral lines $\Re(w) = 3$ and $\Re(u) = 3$ to $\Re(w) = -\frac{1}{5}$ and $\Re(u) = -\frac{1}{5}$ and invoke the Residue Theorem to conclude 
\begin{equation}\label{S1dQ'aleqZ}
\begin{aligned}
    S^{1,d}_{Q'}(a \leq Z) & = \frac{\check{J}(0)\cdot X\log Y}{2\pi^2}L(1,f\otimes g)L(1,\text{Sym}^2f)L(1,\text{Sym}^2g)\mathcal{Z}(0,0;Q') \\
    & + O(X + X(\log X)^{63}/Z),
\end{aligned}
\end{equation}
where the implied constant depends only on the forms $f$ and $g$. In particular, with $\mathcal{Z}(w,u;Q')$'s absolutely convergent on 
$\Re(w), \Re(u) \geq -\frac{1}{4} + \epsilon$, we may directly check as in \cite{Petrow-MR3180602} that 
\[
    \mathcal{Z}(0,0;1) + i^{\kappa_1}\eta_f\mathcal{Z}(0,0;q_1) - i^{\kappa_2}\eta_g\mathcal{Z}(0,0;q_2) -i^{\kappa_1+\kappa_2}\eta_f\eta_g\mathcal{Z}(0,0;q_1q_2) = 0,
\]
hence the main term in the moment \eqref{momentequation} vanishes, if and only if $i^{\kappa_1}\eta_f = -1$ and $q_1 = \square$ or $i^{\kappa_2}\eta_g = 1$ and $q_2 = \square$,
in which case the moment vanishes identically.

It remains to estimate the "off-diagonal" term $S^{1,n}_{Q'}(a \leq Z)$, which is given by 
\begin{equation*}
    \frac{X}{16} \sum_{\substack{(a,2Q)=1 \\ a \leq Z}}\mu(a)\sum_{b|Q}\mu(b)\sum_{k\neq 0}(-1)^k \mathop{\sum\sum}_{ (mn,2a)=1}\frac{\lambda_f(m)\lambda_g(n)}{\sqrt{mn}}\frac{\chi_b(mnQ')G_k(mnQ')}{a^2bmnQ'}\check{J}\bigg(\frac{kX}{16mnQ'a^2b}\bigg)W_1\bigg(\frac{m}{Y}\bigg)W_2\bigg(\frac{n}{Y}\bigg).
\end{equation*}
By Lemma \ref{Gauss sums}, we notice that $G_{4k}(n) = G_{k}(n)$, and thus $S^{1,n}_{Q'}(a\leq Z)$ could be re-written as 
\begin{equation*}
    S^{1,n}_{Q'}(a\leq Z) = \frac{X}{16}\sum_{\substack{(a,2Q)=1 \\ a \leq Z}}\frac{\mu(a)}{a^2}\sum_{b|Q}\frac{\mu(b)}{b}\big\{
        \sum_{l\geq 1}T^{1,n}_{Q'}(l) - T^{1,n}_{Q'}(0)\big\},
\end{equation*}
where for $l\geq 0$
\begin{equation*}
        T^{1,n}_{Q'}(l) := \psum_{(k_1,2)=1}\sum_{\substack{(k_2,2)=1\\ k_2 \geq 1}}\mathop{\sum\sum}_{ (mn,2a)=1}\frac{\lambda_f(m)\lambda_g(n)\chi_b(mnQ')}{\sqrt{mn}}W_1(m/Y)W_2(n/Y)\frac{G_{2^{\delta(l)}k_1k_2^2}(mnQ')}{mnQ'}\check{J}\bigg(\frac{2^lk_1k_2^2X}{16mnQ'a^2b}\bigg),
\end{equation*}
where $\delta(l) = 1$ if $2 \nmid l$ and is $0$ otherwise. Due to the structural similarity of the above $T^{1,n}_{Q'}(l)$'s, we shall only estimate $T^{1,n}_{Q'}(0)$, as the other terms could be treated with minor modifications.

Upon inserting dyadic partitions of unity to the $m$ and $n$ summations, and appealing to the definitions of $W_1(\cdot)$ and $W_2(\cdot)$, we may rewrite $T^{1,n}_{Q'}(0)$ as 
\begin{equation*}
    \begin{aligned}
         T^{1,n}_{Q'}(0) & = \frac{1}{(2\pi i)^2}  \dsum_{N_1}\dsum_{N_2}\int_{(3)}\int_{(3)}
            \gamma_1(w)\gamma_2(u)w^{-1}u^{-2}\bigg(\frac{Y}{2\pi}\bigg)^{w+u}\psum_{(k_1,2)=1}\sum_{\substack{(k_2,2)=1 \\ k_2\geq 1}}\mathop{\sum\sum}_{(mn,2a)=1}\frac{\lambda_f(m)\lambda_g(n)\chi_b(mnQ')}{\sqrt{mn}}     \\
            & \times \frac{G_{k_1k_2^2}(mnQ')}{mnQ'}m^{-w}n^{-u}\check{J}\bigg(\frac{k_1k_2^2X}{16mnQ'a^2b}\bigg)G(m/N_1)G(n/N_2)dwdu.
    \end{aligned}
\end{equation*}
We now define $V(x) = G(x/2) + G(x) + G(2x)$, thus $\text{supp}(G) \subset \text{supp}(V)$. We thus insert $V(m/N_1)$ and $V(n/N_2)$ into the $m,n$-summations, and invoke Mellin inversion on $G(\cdot)$'s to derive 
\begin{equation*}
    \begin{aligned}
         T^{1,n}_{Q'}(0) & = \frac{1}{(2\pi i)^4}  \dsum_{N_1}\dsum_{N_2}\int_{(3)}\int_{(3)}
            \gamma_1(w)\gamma_2(u)(2\pi)^{-(w+u)}\bigg(\frac{Y}{N_1}\bigg)^w
            \bigg(\frac{Y}{N_2}\bigg)^uw^{-1}u^{-2} \\
            & \times \int_{(0)}\int_{(0)}\psum_{(k_1,2)=1}T(k_1,1/2+v_1,1/2+v_2,Q')N_1^{v_1}N_2^{v_2}\tilde{G}(v_1-w)\tilde{G}(v_2-u)dv_2dv_1dwdu,
    \end{aligned}
\end{equation*}
where as in \cite{zhou2025momentderivativesquadratictwists} 
\begin{equation*}
    T(k_1,\alpha,\beta,Q') := \sum_{\substack{(k_2,2)=1 \\ k_2\geq 1}}\mathop{\sum\sum}_{(mn,2a)=1}\frac{\lambda_f(m)\lambda_g(n)\chi_b(mnQ')}{m^{\alpha}n^{\beta}} V\bigg(\frac{m}{N_1}\bigg)V\bigg(\frac{n}{N_2}\bigg)    
    \frac{G_{k_1k_2^2}(mnQ')}{mnQ'}\check{J}\bigg(\frac{k_1k_2^2X}{16mnQ'a^2b}\bigg)
\end{equation*}
for $\Re(\alpha) = \Re(\beta) = 1/2$.

Now, we split the dyadic summation ranges according to $Y$, which break the above into four parts 
\begin{equation*}
    T^{1,n}_{Q'}(0) = T^{1,n}_{Q'}(N_1 \leq Y, N_2\leq Y) + T^{1,n}_{Q'}(N_1 \leq Y, N_2 > Y) + T^{1,n}_{Q'}(N_1 > Y, N_2\leq Y) + T^{1,n}_{Q'}(N_1 > Y, N_2 > Y),
\end{equation*}
where each term corresponds to the contribution from the terms with $N_1,N_2$ in the indicated ranges.
For $T^{1,n}_{Q'}(N_1 \leq Y, N_2\leq Y)$, we 
shift the integral lines $\Re(w) = 3$ and $\Re(u) = 3$ to $\Re(w) = \Re(u) = -6$ while passing through a simple pole at $w=0$ and a double pole at $u=0$, and invoke the Residue Theorem to conclude that 
\begin{equation*}
    \begin{aligned}
        T^{1,n}_{Q'}(N_1 \leq Y, N_2\leq Y) & \ll 
        \dsum_{N_1 \leq Y}\dsum_{N_2\leq Y}\log(Y/N_2)\int^{\infty}_{-\infty}\int^{\infty}_{\infty}\frac{1}{(1+|t_1|)^{20}(1+|t_2|)^{20}}\\
        & \times \bigg|\psum_{(k_1,2)=1}T(k_1,1/2+it_1,1/2+it_2,Q')\bigg|dt_2dt_1.
    \end{aligned}
\end{equation*}
For $T^{1,n}_{Q'}(N_1 \leq Y, N_2 > Y)$, we shift the integral lines for $w,u$ to $\Re(w) = -6$ and $\Re(u) = 6$ while passing through a simple pole at $w = 0$, and invoke the Residue Theorem, getting
\begin{equation*}
    \begin{aligned}
        T^{1,n}_{Q'}(N_1 \leq Y, N_2 > Y) & \ll 
        \dsum_{N_1 \leq Y}\dsum_{N_2 > Y}\bigg(\frac{Y}{N_2}\bigg)^6\int^{\infty}_{-\infty}\int^{\infty}_{-\infty}\int^{\infty}_{-\infty}\frac{1}{(1+|t_1|)^{20}(1+|-6+i(t_2-t_4)|)^{20}(1+|6+it_4|)^{20}}\\
        & \times \bigg|\psum_{(k_1,2)=1}T(k_1,1/2+it_1,1/2+it_2,Q')\bigg|dt_4dt_2dt_1.
    \end{aligned}
\end{equation*}
For $T^{1,n}_{Q'}(N_1 > Y, N_2 \leq Y)$, we shift the integral lines for $w,u$ to $\Re(w) = 6$ and $\Re(u) = -6$ while passing through a double pole at $u = 0$, and invoke the Residue Theorem, obtaining
\begin{equation*}
    \begin{aligned}
        & T^{1,n}_{Q'}(N_1 > Y, N_2 \leq Y) \\
        & \ll 
        \dsum_{N_1 > Y}\dsum_{N_2 \leq Y}\bigg(\frac{Y}{N_1}\bigg)^6\log(Y/N_2)\int^{\infty}_{-\infty}\int^{\infty}_{-\infty}\int^{\infty}_{-\infty}\frac{1}{(1+|t_2|)^{20}(1+|-6+i(t_1-t_3)|)^{20}(1+|6+it_3|)^{20}}\\
        & \times \bigg|\psum_{(k_1,2)=1}T(k_1,1/2+it_1,1/2+it_2,Q')\bigg|dt_3dt_2dt_1.
    \end{aligned}
\end{equation*}
Finally, for $T^{1,n}_{Q'}(N_1 > Y, N_2 > Y)$, we shift the integral lines for $w,u$ to $\Re(w)=\Re(u) = 6$ and invoke the Residue Theorem, reaching
\begin{equation*}
    \begin{aligned}
        & T^{1,n}_{Q'}(N_1 > Y, N_2 > Y)\\
        & \ll 
        \dsum_{N_1 > Y}\dsum_{N_2 > Y}\bigg(\frac{Y}{N_1}\bigg)^6\bigg(\frac{Y}{N_2}\bigg)^6\int^{\infty}_{-\infty}\int^{\infty}_{-\infty}\int^{\infty}_{-\infty}\int^{\infty}_{-\infty}\frac{1}{(1+|-6 + i(t_1-t_3)|)^{20}(1+|-6+i(t_2-t_4)|)^{20}}\\
        & \times \frac{1}{(1+|6+it_3|)^{20}(1+|6+it_4|)^{20}}\bigg|\psum_{(k_1,2)=1}T(k_1,1/2+it_1,1/2+it_2,Q')\bigg|dt_4dt_3dt_2dt_1.
    \end{aligned}
\end{equation*}
Now, by (5.9) of \cite{zhou2025momentderivativesquadratictwists}, we see that 
\begin{equation*}
    \begin{aligned}
        T^{1,n}_{Q'}(0) & \ll \frac{a^2b}{X}\bigg\{\dsum_{N_1\leq Y}\dsum_{N_2\leq Y}\log(Y/N_2)\sqrt{N_1N_2} + \dsum_{N_1\leq Y}\dsum_{N_2>Y}\bigg(\frac{Y}{N_2}\bigg)^6\sqrt{N_1N_2} \\
        & + 
        \dsum_{N_1> Y}\dsum_{N_2\leq Y}\bigg(\frac{Y}{N_1}\bigg)^6\log(Y/N_2)\sqrt{N_1N_2} + 
        \dsum_{N_1> Y}\dsum_{N_2> Y}\bigg(\frac{Y}{N_1}\bigg)^6\bigg(\frac{Y}{N_2}\bigg)^6\sqrt{N_1N_2}
        \bigg\} \\
        & \ll \frac{a^2bY\log(Y)}{X} \ll \frac{a^2b}{(\log X)^{99-\epsilon}},
    \end{aligned}
\end{equation*}
as $Y = \frac{X}{(\log X)^{200}}$.
By similar calculation and section 5.4 of \cite{zhou2025momentderivativesquadratictwists}, we have that for general $l \geq 0$:
\begin{equation*}
    T^{1,n}_{Q'}(l) \ll \frac{a^2b}{2^l(\log X)^{99}}.
\end{equation*}
Therefore we conclude that 
\begin{equation}\label{S1nQ'aleqZ}
    S^{1,n}_{Q'}(a\leq Z) \ll \frac{XZ}{(\log X)^{99}}.
\end{equation}
Finally, upon taking $Z = (\log X)^{80}$, we conclude from \eqref{S1Q'a>Z}, \eqref{S1dQ'aleqZ} and \eqref{S1nQ'aleqZ} that 
\[
    S^1_{Q'} = \frac{\check{J}(0)L(1,f\otimes g)L(1, \text{Sym}^2 f)L(1,\text{Sym}^2 g)\mathcal{Z}(0,0;Q')}{2\pi^2}\cdot X\log X + O_{f,g}(X)
\]
where $\epsilon > 0$ could be arbitrarily small.
This completes the proof of Proposition 1. $\square$

\section{Estimating $\mathfrak{I}_2$}
In this section we treat $\mathfrak{I}_2$ and prove Proposition 2.

We recall that 
\[
    \mathfrak{I}_2 = S^2_1 + i^{\kappa_1}\eta_fS^2_{q_1} - i^{\kappa_2}\eta_gS^2_{q_2} - i^{\kappa_1 + \kappa_2}\eta_f\eta_gS^2_{q_1q_2}
\]
where
\[
    S^2_{Q'} := \psum_{\substack{(d,2Q) = 1}}\sum_{m}\sum_{n}\frac{\lambda_f(m)\lambda_g(n)}{\sqrt{mn}}\chi_{8d}(mnQ')W_1\bigg(\frac{m}{Y}\bigg)\bigg(W_2\bigg(\frac{n}{8d\sqrt{q_2}}\bigg) - W_2\bigg(\frac{n}{Y}\bigg)\bigg)J\bigg(\frac{8d}{X}\bigg)
\] and $Q = q_1q_2$.
Using Mobius to drop the square-free condition on $d$, the above becomes
\begin{equation*}
    \begin{aligned}
        S^2_{Q'} & = \sum_{(a,2Q)=1}\mu(a)\sum_{(d,2Q)=1}\mathop{\sum\sum}_{(mn,2a)=1}\frac{\lambda_f(m)\lambda_g(n)\chi_{8d}(mnQ')}{\sqrt{mn}} W_1\bigg(\frac{m}{Y}\bigg)\bigg(W_2\bigg(\frac{n}{8a^2d\sqrt{q_2}}\bigg) - W_2\bigg(\frac{n}{Y}\bigg)\bigg)J\bigg(\frac{8a^2d}{X}\bigg) \\
        & = S^2_{Q'}(a\leq Z) + S^2_{Q'}(a > Z)
    \end{aligned}
\end{equation*}
where as before $Z = (\log X)^{80}$, and $S^2_{Q'}(a\leq Z)$ and $S^2_{Q'}(a > Z)$ denote respectively the contribution from the terms with $a \leq Z$ and $a > Z$ in the above. 

By Proposition 3.1 of \cite{Kumar.etc-MR4765788} and Proposition 6.2 of \cite{Li-MR4768632} as well as Cauchy-Schwarz, we may bound $S^2_{Q'}(a > Z)$ by 
\begin{equation}\label{S2Q'a>Z}
    S^2_{Q'}(a > Z) \ll \frac{X(\log X)^{3/2+\epsilon}}{Z^{1-\epsilon}}.
\end{equation}

Now we treat $S^2_{Q'}(a \leq Z)$. We drop the co-prime condition on the $d$-summation and write the above as
\begin{equation*}
    \sum_{\substack{(a,2Q)=1 \\ a \leq Z}}\mu(a)\sum_{b|Q}\mu(b)\sum_{(d,2)=1}\mathop{\sum\sum}_{(mn,2a)=1}\frac{\lambda_f(m)\lambda_g(n)\chi_{8bd}(mnQ')}{\sqrt{mn}} W_1\bigg(\frac{m}{Y}\bigg)\bigg(W_2\bigg(\frac{n}{8a^2bd\sqrt{q_2}}\bigg) - W_2\bigg(\frac{n}{Y}\bigg)\bigg)J\bigg(\frac{8a^2bd}{X}\bigg)
\end{equation*}
Applying Poisson summation, we transform $S^2_{Q'}(a\leq Z)$ as
\begin{equation*}
       \begin{aligned}
        S^{2}_{Q'}(a \leq Z) & =
        \frac{X}{16}\sum_{\substack{(a,2Q)=1 \\ a \leq Z}}\frac{\mu(a)}{a^2}\sum_{b|Q}\frac{\mu(b)}{b}(2\pi i)^{-2}\int_{(3)}\int_{(3)}\frac{\gamma_1(w)}{w}\frac{\gamma_2(u)}{u^2}(2\pi)^{-(w+u)}Y^w\\
        & \times \mathop{\sum\sum}_{(mn,2a)=1}\frac{\lambda_f(m)\lambda_g(n)\chi_b(mnQ')}{m^{1/2+w}n^{1/2+u}}\sum_{k\in \mathbb{Z}}(-1)^k\frac{G_k(mnQ')}{mnQ'}\check{H}_u\bigg(\frac{kX}{16a^2bmnQ'}\bigg)dudw
    \end{aligned}
\end{equation*}
in which $H_u(s) := J(s)\cdot \big((s\sqrt{q_2}X)^u - Y^u\big)$.
As before, we denote the contribution from the term with $k=0$ and those with $k \neq 0$ respectively by $S^{2,d}_{Q'}(a\leq Z)$ and $S^{2,n}_{Q'}(a\leq Z)$, and treat them separately. 

For the diagonal term $S^{2,d}_{Q'}(a\leq Z)$, we use \eqref{asymptotics for Mobius} and proceed similarly as before to write it as 
\begin{equation*}
    \begin{aligned}
         S^{2,d}_{Q'}(a\leq Z) & = \frac{X}{16}\sum_{\substack{(a,2Q)=1 \\ a \leq Z}}\frac{\mu(a)}{a^2}\sum_{b|Q}\frac{\mu(b)}{b}(2\pi i)^{-2}\int_{(3)}\int_{(3)}\frac{\gamma_1(w)}{w}\frac{\gamma_2(u)}{u^2}(2\pi)^{-(w+u)}Y^w\\
         & \times \mathop{\sum\sum}_{\substack{(mn,2a)=1 \\ mnQ' = \square}}\frac{\lambda_f(m)\lambda_g(n)}{m^{1/2+w}n^{1/2+u}}\frac{\phi(mnQ')}{mnQ'}\check{H}_u(0)dudw \\
         & = \frac{X}{2\pi^2}(2\pi i)^{-2}\int_{-\infty}^{\infty}J(x)\int_{(3)}\int_{(3)}\frac{\gamma_1(w)}{w}\frac{\gamma_2(u)}{u^2}\big((x\sqrt{q_2}X/Y)^u-1\big)\bigg(\frac{Y}{2\pi}\bigg)^{u+w}\\
         & \times L(1+w+u,f\otimes g)         L(1+2w,\text{Sym}^2 f)L(1+2u,\text{Sym}^2g)\mathcal{Z}(w,u;Q')dudwdx + O(X(\log X)^{63}/Z).
    \end{aligned}
\end{equation*}
Noticing the simple poles of the integrand at $w = 0$ and $u=0$, we shift the integral lines for $w,u$ to $\Re(w) = \Re(u) = -\frac{1}{5}$ and invoke the Residue Theorem to conclude that 
\begin{equation}\label{S2dQ'aleqZ}
    S^{2,d}_{Q'}(a \leq Z) \ll X(
    \log \log X) + X(\log X)^{63}/Z.
\end{equation}

Next we estimate the off-diagonal term $S^{2,n}_{Q'}(a\leq Z)$. Similar as in the previous section, the off-diagonal term could be rewritten as
\[
    S^{2,n}_{Q'}(a\leq Z) = \frac{X}{16}\sum_{\substack{(a,2Q)=1 \\ a\leq Z}}\frac{\mu(a)}{a^2}\sum_{b|Q}\frac{\mu(b)}{b}\bigg\{ 
    \sum_{l\geq 1}T^{2,n}_{Q'}(l) - T^{2,n}_{Q'}(0)
    \bigg\},
\]
in which for $l \geq 0$ 
\begin{equation*}
    \begin{aligned}
         T^{2,n}_{Q'}(l) &:= \psum_{(k_1,2)=1}\sum_{\substack{(k_2,2)=1\\ k_2 \geq 1}}\mathop{\sum\sum}_{ (mn,2a)=1}(2\pi i)^{-2}\int_{(3)}\int_{(3)}\frac{\gamma_1(w)}{w}\frac{\gamma_2(u)}{u^2}(2\pi) ^{-(w+u)}Y^w\frac{\lambda_f(m)\lambda_g(n)\chi_b(mnQ')}{m^{1/2+w}n^{1/2+u}}\\
      & \times \frac{G_{2^{\delta(l)}k_1k_2^2}(mnQ')}{mnQ'}\check{H}_u\bigg(\frac{2^lk_1k_2^2X}{16a^2bmnQ'}\bigg)dudw.
    \end{aligned}
\end{equation*}
We shall work in full detail only for $T^{2,n}_{Q'}(0)$. As before, we dyadically partition $m,n$-summations using $G(\cdot)$'s, then insert $V(\cdot)$'s and invoke the Mellin inversion for $G(\cdot)$'s to arrive at 
\begin{equation*}
    \begin{aligned}
        T^{2,n}_{Q'}(0) & = \dsum_{N_1}\dsum_{N_2}\psum_{(k_1,2)=1}\sum_{\substack{(k_2,2)=1 \\ k_2 \geq 1}}\mathop{\sum\sum}_{(mn,2a)=1}(2\pi i)^{-5}\int_{(1/2)}\int^{\infty}_0J(x)x^{-s}\int_{(3)}\int_{(3)}\int_{(0)}\int_{(0)}\frac{\gamma_1(w)}{w}\frac{\gamma_2(u)}{u}(2\pi)^{-(w+u)}\\
        & \times \bigg(\frac{Y}{N_1}\bigg)^{w}\bigg(\frac{Y}{N_2}\bigg)^{u}\frac{\big((x\sqrt{q_2}X/Y)^u - 1\big)}{u} \frac{\lambda_f(m)\lambda_g(n)\chi_b(mnQ')}{m^{1/2+v_1}n^{1/2+v_2}}\cdot \frac{G_{k_1k_2^2}(mnQ')}{mnQ'}V(m/N_1)V(n/N_2)\\
        & \times \Gamma(s)(\cos(\pi s/2) + \text{sgn}(k_1)\sin(\pi s/2))\bigg(\frac{8a^2bmnQ'}{\pi |k_1|k_2^2X}\bigg)^s\tilde{G}(v_1-w)\tilde{G}(v_2-u)N_1^{v_1}N_2^{v_2}dv_2dv_1dudwdxds \\
        & = T^{2,n}_{Q'}(N_1 \leq Y, N_2\leq Y) +
        T^{2,n}_{Q'}(N_1 \leq Y, Y< N_2 \leq X) +
        T^{2,n}_{Q'}(N_1\leq Y, X< N_2 )\\
        & + 
        T^{2,n}_{Q'}(Y< N_1, N_2 \leq Y)
         + T^{2,n}_{Q'}(Y<N_1, Y < N_2 \leq X) + T^{2,n}_{Q'}(Y < N_1, X < N_2),
    \end{aligned}
\end{equation*}
where the last six terms correspond to the contribution to $T^{2,n}_{Q'}(0)$ from those terms with the corresponding ranges for $N_1,N_2$. 

For $T^{2,n}_{Q'}(N_1 \leq Y, N_2\leq Y)$, we shift the integral lines for $w,u$ to $\Re(w) = \Re(u) = -6$ passing through simple poles at $w = 0$ and $u=0$ ,and then apply the Residue Theorem to conclude that 
\begin{equation*}
    \begin{aligned}
        T^{2,n}_{Q'}(N_1 \leq Y, N_2\leq Y) & \ll (\log \log X)\dsum_{N_1\leq Y}\dsum_{N_2\leq Y}\int_{-\infty}^{\infty}\int^{\infty}_{-\infty}\frac{1}{(1+|t_1|)^{20}(1+|t_2|)^{20}}\\
        & \times \bigg|\psum_{(k_1,2)=1}T(k_1,1/2+it_1,1/2+it_2,Q')\bigg|dt_2dt_1.
    \end{aligned}
\end{equation*}

For $T^{2,n}_{Q'}(Y<N_1 , N_2\leq Y)$, we shift the integral lines for $w,u$ to $\Re(u) = -6$ and $\Re(w) = 6$, by the Residue Theorem we have 
\begin{equation*}
    \begin{aligned}
        & T^{2,n}_{Q'}(Y<N_1, N_2 \leq Y)\\ 
        & \ll (\log \log X)\mathop{\dsum\dsum}_{\substack{N_1> Y \\ N_2 \leq Y}}(Y/N_1)^{6}\int_{-\infty}^{\infty}\int_{-\infty}^{\infty}\int_{-\infty}^{\infty}\frac{1}{(1+|t_2|)^{20}(1+|6 + it_3|)^{20}}\\
        & \times \frac{1}{(1+|-6 + i(t_1-t_3)|)^{20}} \bigg|\psum_{(k_1,2)=1}T(k_1,1/2+it_1,1/2+it_2,Q')\bigg|dt_4dt_2dt_1.
    \end{aligned}
\end{equation*}

To estimate the other four terms, we firstly follow (6.1) of \cite{zhou2025momentderivativesquadratictwists} to define the following "shifted" version of $T(k_1,1/2 + it_1,1/2+it_2,Q')$ as
\begin{equation*}
    \begin{aligned}
        & T(k_1,1/2+it_1, 1/2+it_2,Q', c + it_4)\\  
        & := \frac{1}{2\pi i}\int_{(1/2)}\Gamma(s)(\cos(\pi s/2) + \text{sgn}(k_1)\sin(\pi s/2))\bigg(\frac{8a^2bQ'}{\pi|k_1|X}\bigg)^s\\
        & \times \sum_{\substack{(k_2,2)=1 \\ k_2\geq 1}}\mathop{\sum\sum}_{(mn,2a)=1}\frac{\lambda_f(m)\lambda_g(n)\chi_b(mnQ')}{m^{1/2+v_1-s}n^{1/2+v_2-s}k_2^{2s}}\cdot \frac{G_{k_1k_2^2}(mnQ')}{mnQ'}V(m/N_1)V(n/N_2)
    \end{aligned}
\end{equation*}
for $c,t_1,t_2,t_4 \in \mathbb{R}$.

Now for $T^{2,n}_{Q'}(N_1 \leq Y, Y < N_2\leq X)$, we shift the integral lines for $w,u$ to $\Re(w) = -6$ and $\Re(u) = 1/\log X$, by the Residue Theorem we have 
\begin{equation*}
    \begin{aligned}
        & T^{2,n}_{Q'}(N_1\leq Y, Y < N_2 \leq X)\\ 
        & \ll (\log \log X)^2\mathop{\dsum\dsum}_{\substack{N_1\leq Y \\ Y < N_2 \leq X}}(Y/N_2)^{(\log X)^{-1}}\int_{-\infty}^{\infty}\int_{-\infty}^{\infty}\int_{-\infty}^{\infty}\frac{1}{(1+|t_1|)^{20}(1+|(\log X)^{-1} + it_4|)^{20}}\\
        & \times \frac{1}{(1+|-(\log X)^{-1} + i(t_2-t_4)|)^{20}} \bigg|\psum_{(k_1,2)=1}T(k_1,1/2+it_1,1/2+it_2,Q',(\log X)^{-1}+it_4)\bigg|dt_4dt_2dt_1.
    \end{aligned}
\end{equation*}
In the above we have used the fact that $\frac{(xX/Y)^u - 1}{u^2}\ll \log(X/Y)^2 \ll (\log \log X)^2$.

For $T^{2,n}_{Q'}(N_1 \leq Y, X < N_2)$, we shift the integral lines for $w,u$ to $\Re(w) = -6$ and $\Re(u) = 6$, by the Residue Theorem we have 
\begin{equation*}
    \begin{aligned}
        & T^{2,n}_{Q'}(N_1\leq Y, X < N_2 )\\ 
        & \ll \mathop{\dsum\dsum}_{\substack{N_1\leq Y \\ X < N_2}}(Y/N_2)^6\int_{-\infty}^{\infty}\int_{-\infty}^{\infty}\int_{-\infty}^{\infty}\frac{1}{(1+|t_1|)^{20}(1+|6 + it_4|)^{20}(1+|-6 + i(t_2-t_4)|)^{20}}\\
        & \times \bigg|\psum_{(k_1,2)=1}T(k_1,1/2+it_1,1/2+it_2,Q',6+it_4)\bigg|dt_4dt_2dt_1.
    \end{aligned}
\end{equation*}

For $T^{2,n}_{Q'}(Y<N_1 , Y< N_2\leq X)$, we shift the integral lines for $w,u$ to $\Re(u) = (\log X)^{-1}$ and $\Re(w) = 6$, by the Residue Theorem we have 
\begin{equation*}
    \begin{aligned}
        & T^{2,n}_{Q'}(Y<N_1, Y< N_2 \leq X)\\ 
        & \ll (\log \log X)^2\mathop{\dsum\dsum}_{\substack{N_1> Y \\ Y < N_2 \leq X}}(Y/N_1)^{6}(Y/N_2)^{(\log X)^{-1}}\int^{\infty}_{\infty}\int_{-\infty}^{\infty}\int_{-\infty}^{\infty}\int_{-\infty}^{\infty}\frac{1}{(1+|(\log X)^{-1}+it_4|)^{20}(1+|6 + it_3|)^{20}}\\
        & \times \frac{1}{(1+|-6 + i(t_1-t_3)|)^{20}}\frac{1}{(1+|-(\log X)^{-1} + i(t_2-t_4)|)^{20}}\\
        &\times \bigg|\psum_{(k_1,2)=1}T(k_1,1/2+it_1,1/2+it_2,Q', (\log X)^{-1} + it_4)\bigg|dt_4dt_3dt_2dt_1.
    \end{aligned}
\end{equation*}

Finally, for $T^{2,n}_{Q'}(Y<N_1 , X < N_2)$, we shift the integral lines for $w,u$ to $\Re(u) = 6$ and $\Re(w) = 6$, by the Residue Theorem we have 
\begin{equation*}
    \begin{aligned}
        & T^{2,n}_{Q'}(Y<N_1, X< N_2)\\ 
        & \ll \mathop{\dsum\dsum}_{\substack{N_1> Y \\ N_2 > X }}(Y/N_1)^{6}(Y/N_2)^{6}\int^{\infty}_{\infty}\int_{-\infty}^{\infty}\int_{-\infty}^{\infty}\int_{-\infty}^{\infty}\frac{1}{(1+|6+it_4|)^{20}(1+|6 + it_3|)^{20}}\\
        & \times \frac{1}{(1+|-6 + i(t_1-t_3)|)^{20}}\frac{1}{(1+|-6 + i(t_2-t_4)|)^{20}}\\
        &\times \bigg|\psum_{(k_1,2)=1}T(k_1,1/2+it_1,1/2+it_2,Q', 6 + it_4)\bigg|dt_4dt_3dt_2dt_1.
    \end{aligned}
\end{equation*}

Now, by (5.9) of \cite{zhou2025momentderivativesquadratictwists} we conclude that 
\begin{equation*}
    \begin{aligned}
    T^{2,n}_{Q'}(0) & \ll \frac{a^2b}{X}\bigg\{
    (\log \log X)\mathop{\dsum\dsum}_{\substack{N_1 \leq Y \\ N_2\leq Y}}\sqrt{N_1N_2} 
    + (\log \log X)^2\mathop{\dsum\dsum}_{\substack{N_1 \leq Y \\ Y < N_2\leq X}}(Y/N_2)^{(\log X)^{-1}}\sqrt{N_1N_2} \\
    & + \mathop{\dsum\dsum}_{\substack{N_1 \leq Y \\ X<N_2}}(Y/N_2)^6\sqrt{N_1N_2} 
    + (\log \log X)\mathop{\dsum\dsum}_{\substack{N_1 > Y \\ N_2\leq Y}}(Y/N_1)^6\sqrt{N_1N_2}\\
    & +(\log \log X)^2\mathop{\dsum\dsum}_{\substack{N_1 > Y \\ Y < N_2\leq X}}(Y/N_1)^6(Y/N_1)^{(\log X)^{-1}}\sqrt{N_1N_2}
    + \mathop{\dsum\dsum}_{\substack{N_1 > Y \\ N_2>X}}(Y/N_1)^6(Y/N_2)^6\sqrt{N_1N_2}
    \bigg\} \\
    & \ll \frac{a^2b}{X}\sqrt{YX}(\log \log X)^3 \ll \frac{a^2b}{(\log X)^{99-\epsilon}}.
    \end{aligned}
\end{equation*}
By a similar process one can also get for general $l \geq 0$
\[
    T^{2,n}_{Q'}(l) \ll \frac{a^2b}{2^l (\log X)^{99-\epsilon}},
\]
and thus 
\begin{equation}\label{S2nQ'aleqZ}
    S^{2,n}_{Q'}(a\leq Z) \ll \frac{XZ}{(\log X)^{99-\epsilon}} \ll X/(\log X)^{10},
\end{equation}
with $Z = (\log X)^{80}$.

Thus, combining \eqref{S2Q'a>Z}, \eqref{S2dQ'aleqZ} and \eqref{S1nQ'aleqZ}, we conclude that 
\[
    S^{2}_{Q'} \ll X(\log \log X),
\]
finishing the proof of Proposition 2. $\square$

\section{Estimating $\mathfrak{I}_3$}
In this section we treat $\mathfrak{I}_3$ and prove Proposition 3.

Due to the structural similarity of $\mathfrak{I}_3$ and $\mathfrak{I}_2$ treated in the previous section, we shall be concise and only provide the sketch of the proof. The notations in this section shall follow a same fashion as those in the previous two sections.

As before, we define
\[
    S^3_{Q'} := \psum_{\substack{(d,2Q) = 1}}\sum_{m}\sum_{n}\frac{\lambda_f(m)\lambda_g(n)}{\sqrt{mn}}\chi_{8d}(mnQ')\bigg(W_1\bigg(\frac{m}{8d\sqrt{q_1}}\bigg) - W_1\bigg(\frac{m}{Y}\bigg)\bigg)W_2\bigg(\frac{n}{Y}\bigg)J\bigg(\frac{8d}{X}\bigg)
\] and $Q = q_1q_2$,
thus 
\[
    \mathfrak{I}_3 = S^3_{1} + i^{\kappa_1}\eta_fS^3_{q_1} - i^{\kappa_2}\eta_gS^3_{q_2} - i^{\kappa_1 + \kappa_2}\eta_f\eta_gS^3_{q_1q_2}.
\]

As before, we drop the square-free condition as
\begin{equation*}
    \begin{aligned}
        S^3_{Q'} & = \sum_{(a,2Q)=1}\mu(a)\sum_{(d,2Q)=1}\mathop{\sum\sum}_{(mn,2a)=1}\frac{\lambda_f(m)\lambda_g(n)\chi_{8d}(mnQ')}{\sqrt{mn}} W_2\bigg(\frac{n}{Y}\bigg)\bigg(W_1\bigg(\frac{m}{8a^2d\sqrt{q_1}}\bigg) - W_1\bigg(\frac{m}{Y}\bigg)\bigg)J\bigg(\frac{8a^2d}{X}\bigg) \\
        & = S^3_{Q'}(a\leq Z) + S^3_{Q'}(a > Z)
    \end{aligned}
\end{equation*}
where $S^3_{Q'}(a\leq Z)$ and $S^3_{Q'}(a > Z)$ are the contribution from the terms with $a \leq Z$ and $a > Z$ respectively. For $S^3_{Q'}(a >Z)$, we use Proposition 6.1 of \cite{Li-MR4768632} and Proposition 3.2 of \cite{Kumar.etc-MR4765788} as well as Cauchy-Schwarz to bound it as
\begin{equation}\label{S3Q'a>Z}
    S^3_{Q'}(a > Z) \ll \frac{X(\log X)^{\frac{3}{2}+\epsilon}}{Z^{1-\epsilon}}. 
\end{equation}

Then, as before for $S^3_{Q'}(a \leq Z)$ we apply Poisson summation to transform it as
\begin{equation*}
    \begin{aligned}
        S^3_{Q'}(a\leq Z) & = \frac{X}{16}\sum_{\substack{(a,2Q)=1 \\ a \leq Z}}\frac{\mu(a)}{a^2}\sum_{b|Q}\frac{\mu(b)}{b}(2\pi i)^{-2}\int_{(3)}\int_{(3)}\frac{\gamma_1(w)}{w}\frac{\gamma_2(u)}{u^2}(2\pi)^{-(w+u)}Y^u\\
        & \times \mathop{\sum\sum}_{(mn,2a)=1}\frac{\lambda_f(m)\lambda_g(n)\chi_b(mnQ')}{m^{1/2+w}n^{1/2+u}}\sum_{k\in \mathbb{Z}}(-1)^k\frac{G_k(mnQ')}{mnQ'}\check{H}_w\bigg(\frac{kX}{16a^2bmnQ'}\bigg)dudw
    \end{aligned}
\end{equation*}
where $H_w(x):= J(s)\big((s\sqrt{q_1}X)^w - Y^w\big)$.
Now for the diagonal term $S^{3,d}_{Q'}(a\leq Z)$, which correspond the contribution from the $k=0$ terms in $S^{3}_{Q'}(a\leq Z)$, we as before use \eqref{asymptotics for Mobius} and estimate it as
\begin{equation}\label{S3dQ'aleqZ}
    \begin{aligned}
         S^{3,d}_{Q'}(a\leq Z) & 
          = \frac{X}{2\pi^2}(2\pi i)^{-2}\int_{-\infty}^{\infty}J(x)\int_{(3)}\int_{(3)}\frac{\gamma_1(w)}{w}\frac{\gamma_2(u)}{u^2}\big((x\sqrt{q_1}X/Y)^w-1\big)\bigg(\frac{Y}{2\pi}\bigg)^{u+w}\\
         & \times L(1+w+u,f\otimes g)         L(1+2w,\text{Sym}^2 f)L(1+2u,\text{Sym}^2g)\mathcal{Z}(w,u;Q')dudwdx \\
         & + O(X(\log X)^{63}/Z)\\
         & \ll XY^{-\frac{1}{5}}\log Y + X(\log X)^{63}/Z.
    \end{aligned}
\end{equation}
after shifting the integral lines to $\Re(w) = \Re(u) = -\frac{1}{5}$ while passing through a double pole only at $u = 0$.

As for the off-diagonal term $S^{3,n}_{Q'}(a\leq Z) := S^{3}_{Q'}(a\leq Z) - S^{3,d}_{Q'}(a\leq Z)$, we proceed as before to write it as
\begin{equation*}
    S^{3,n}_{Q'}(a\leq Z) = \frac{X}{16}\cdot \sum_{\substack{(a,2Q)=1 \\ a\leq Z}}\frac{\mu(a)}{a^2}\sum_{b|Q}\frac{\mu(b)}{b}\bigg\{
    \sum_{l\geq 1}T^{3,n}_{Q'}(l) - T^{3,n}_{Q'}(0)
    \bigg\}
\end{equation*}
in which for $l \geq 0$ 
\begin{equation*}
    \begin{aligned}
         T^{3,n}_{Q'}(l) &:= \psum_{(k_1,2)=1}\sum_{\substack{(k_2,2)=1\\ k_2 \geq 1}}\mathop{\sum\sum}_{ (mn,2a)=1}(2\pi i)^{-2}\int_{(3)}\int_{(3)}\frac{\gamma_1(w)}{w}\frac{\gamma_2(u)}{u^2}(2\pi) ^{-(w+u)}Y^u\frac{\lambda_f(m)\lambda_g(n)\chi_b(mnQ')}{m^{1/2+w}n^{1/2+u}}\\
      & \times \frac{G_{2^{\delta(l)}k_1k_2^2}(mnQ')}{mnQ'}\check{H}_w\bigg(\frac{2^lk_1k_2^2X}{16a^2bmnQ'}\bigg)dudw.
    \end{aligned}
\end{equation*}

As before we focus on estimating $T^{3,n}_{Q'}(0)$.
Dyadically partition $m,n$-summations using $G(\cdot)$'s then inserting $V(\cdot)$'s and invoking the Mellin inversion for $G(\cdot)$'s, we arrive at 
\begin{equation*}
    \begin{aligned}
        T^{3,n}_{Q'}(0) & = \dsum_{N_1}\dsum_{N_2}\psum_{(k_1,2)=1}\sum_{\substack{(k_2,2)=1 \\ k_2 \geq 1}}\mathop{\sum\sum}_{(mn,2a)=1}(2\pi i)^{-5}\int_{(1/2)}\int^{\infty}_0J(x)x^{-s}\int_{(3)}\int_{(3)}\int_{(0)}\int_{(0)}\gamma_1(w)\frac{\gamma_2(u)}{u^2}(2\pi)^{-(w+u)}\\
        & \times \bigg(\frac{Y}{N_1}\bigg)^{w}\bigg(\frac{Y}{N_2}\bigg)^{u}\frac{\big((x\sqrt{q_1}X/Y)^w - 1\big)}{w} \frac{\lambda_f(m)\lambda_g(n)\chi_b(mnQ')}{m^{1/2+v_1}n^{1/2+v_2}}\cdot \frac{G_{k_1k_2^2}(mnQ')}{mnQ'}V(m/N_1)V(n/N_2)\\
        & \times \Gamma(s)(\cos(\pi s/2) + \text{sgn}(k_1)\sin(\pi s/2))\bigg(\frac{8a^2bmnQ'}{\pi |k_1|k_2^2X}\bigg)^s\tilde{G}(v_1-w)\tilde{G}(v_2-u)N_1^{v_1}N_2^{v_2}dv_2dv_1dudwdxds \\
        & = T^{3,n}_{Q'}(N_1 \leq Y, N_2\leq Y) +
        T^{3,n}_{Q'}(Y < N_1 \leq X, N_2 \leq Y) +
        T^{3,n}_{Q'}(N_1> X, N_2 \leq Y)\\
        & + 
        T^{3,n}_{Q'}(N_1\leq Y, N_2 > Y)
         + T^{3,n}_{Q'}(Y<N_1 \leq X, N_2 > Y) + T^{3,n}_{Q'}(X < N_1, N_2 > Y),
    \end{aligned}
\end{equation*}
where the last six terms correspond to the contribution to $T^{3,n}_{Q'}(0)$ from those terms with the corresponding ranges for $N_1,N_2$. 

For $T^{3,n}_{Q'}(N_1 \leq Y, N_2\leq Y)$, we shift the integral lines for $w,u$ to $\Re(w) = \Re(u) = -6$ passing through a double pole at $u=0$ ,and then apply the Residue Theorem to conclude that 
\begin{equation*}
    \begin{aligned}
        T^{3,n}_{Q'}(N_1 \leq Y, N_2\leq Y) & \ll \dsum_{N_1\leq Y}\dsum_{N_2\leq Y}\log(Y/N_2)\int^{\infty}_{-\infty}\int_{-\infty}^{\infty}\int^{\infty}_{-\infty}\frac{1}{(1+|6+ i(t_1-t_3)|)^{20}(1+|t_2|)^{20}}\\
        & \times \frac{1}{(1+|-6+it_3|)^{20}}\bigg|\psum_{(k_1,2)=1}T(k_1,1/2+it_1,1/2+it_2,Q',-6+it_3)\bigg|dt_3dt_2dt_1.
    \end{aligned}
\end{equation*}

For $T^{3,n}_{Q'}(Y < N_1 \leq X, N_2\leq Y)$, we shift the integral lines for $w,u$ to $\Re(w) = (\log X)^{-1}$ and $\Re(u) = -6$, by the Residue Theorem we have 
\begin{equation*}
    \begin{aligned}
        & T^{3,n}_{Q'}(Y<N_1\leq X, N_2 \leq Y)\\ 
        & \ll \mathop{\dsum\dsum}_{\substack{Y < N_1 \leq X \\ N_2 \leq Y}}(Y/N_1)^{(\log X)^{-1}}\log(Y/N_2)\int_{-\infty}^{\infty}\int_{-\infty}^{\infty}\int_{-\infty}^{\infty}\frac{1}{(1+|t_2|)^{20}(1+|(\log X)^{-1} + it_3|)^{20}}\\
        & \times \frac{1}{(1+|-(\log X)^{-1} + i(t_1-t_3)|)^{20}} \bigg|\psum_{(k_1,2)=1}T(k_1,1/2+it_1,1/2+it_2,Q', (\log X)^{-1} + it_3)\bigg|dt_3dt_2dt_1.
    \end{aligned}
\end{equation*}

For $T^{3,n}_{Q'}(N_1>X, N_2\leq Y)$, we shift the integral lines for $w,u$ to $\Re(w) = 6$ and $\Re(u) = -6$, by the Residue Theorem we have 
\begin{equation*}
    \begin{aligned}
        & T^{3,n}_{Q'}(N_1\leq Y, Y < N_2 \leq X)\\ 
        & \ll \mathop{\dsum\dsum}_{\substack{N_1> X \\ N_2 \leq Y}}(Y/N_1)^{6}\log(Y/N_1) \int_{-\infty}^{\infty}\int_{-\infty}^{\infty}\int_{-\infty}^{\infty}\frac{1}{(1+|t_2|)^{20}(1+|6 + it_3|)^{20}}\\
        & \times \frac{1}{(1+|-6 + i(t_1-t_3)|)^{20}} \bigg|\psum_{(k_1,2)=1}T(k_1,1/2+it_1,1/2+it_2,Q', 6+it_3)\bigg|dt_3dt_2dt_1.
    \end{aligned}
\end{equation*}

For $T^{3,n}_{Q'}(N_1 \leq Y, Y < N_2)$, we shift the integral lines for $w,u$ to $\Re(w) = -6$ and $\Re(u) = 6$, by the Residue Theorem we have 
\begin{equation*}
    \begin{aligned}
        & T^{3,n}_{Q'}(N_1\leq Y, Y < N_2 )\\ 
        & \ll \mathop{\dsum\dsum}_{\substack{N_1\leq Y \\ Y < N_2}}(Y/N_2)^6\int_{-\infty}^{\infty}\int_{-\infty}^{\infty}\int_{-\infty}^{\infty}\frac{1}{(1+|-6+it_3|)^{20}(1+|6 + it_4|)^{20}(1+|-6 + i(t_2-t_4)|)^{20}}\\
        & \times \frac{1}{(1+|6+i(t_1-t_3)|)^{20}}\bigg|\psum_{(k_1,2)=1}T(k_1,1/2+it_1,1/2+it_2,Q',-6+it_3)\bigg|dt_4dt_3dt_2dt_1.
    \end{aligned}
\end{equation*}

For $T^{3,n}_{Q'}(Y<N_1 \leq X , Y< N_2)$, we shift the integral lines for $w,u$ to $\Re(u) = 6$ and $\Re(w) = (\log X)^{-1}$, by the Residue Theorem we have 
\begin{equation*}
    \begin{aligned}
        & T^{3,n}_{Q'}(Y<N_1\leq X, Y< N_2)\\ 
        & \ll \mathop{\dsum\dsum}_{\substack{N_1> Y \\ Y < N_2 \leq X}}(Y/N_2)^{6}(Y/N_1)^{(\log X)^{-1}}\int^{\infty}_{\infty}\int_{-\infty}^{\infty}\int_{-\infty}^{\infty}\int_{-\infty}^{\infty}\frac{1}{(1+|(\log X)^{-1}+it_3|)^{20}(1+|6 + it_4|)^{20}}\\
        & \times \frac{1}{(1+|-6 + i(t_2-t_4)|)^{20}}\frac{1}{(1+|-(\log X)^{-1} + i(t_1-t_3)|)^{20}}\\
        &\times \bigg|\psum_{(k_1,2)=1}T(k_1,1/2+it_1,1/2+it_2,Q', (\log X)^{-1} + it_3)\bigg|dt_4dt_3dt_2dt_1.
    \end{aligned}
\end{equation*}

Finally, for $T^{3,n}_{Q'}(X<N_1 , Y < N_2)$, we shift the integral lines for $w,u$ to $\Re(u) = 6$ and $\Re(w) = 6$, by the Residue Theorem we have 
\begin{equation*}
    \begin{aligned}
        & T^{3,n}_{Q'}(X<N_1, Y< N_2)\\ 
        & \ll \mathop{\dsum\dsum}_{\substack{N_1> X \\ N_2 > Y }}(Y/N_1)^{6}(Y/N_2)^{6}\int^{\infty}_{\infty}\int_{-\infty}^{\infty}\int_{-\infty}^{\infty}\int_{-\infty}^{\infty}\frac{1}{(1+|6+it_4|)^{20}(1+|6 + it_3|)^{20}}\\
        & \times \frac{1}{(1+|-6 + i(t_1-t_3)|)^{20}}\frac{1}{(1+|-6 + i(t_2-t_4)|)^{20}}\\
        &\times \bigg|\psum_{(k_1,2)=1}T(k_1,1/2+it_1,1/2+it_2,Q', 6 + it_3)\bigg|dt_4dt_3dt_2dt_1.
    \end{aligned}
\end{equation*}

Now, by (5.9) of \cite{zhou2025momentderivativesquadratictwists} we conclude that 
\begin{equation*}
    \begin{aligned}
    T^{3,n}_{Q'}(0) & \ll \frac{a^2b}{X}\bigg\{
    \log Y\mathop{\dsum\dsum}_{\substack{N_1 \leq Y \\ N_2\leq Y}}\sqrt{N_1N_2} 
    + \log Y\mathop{\dsum\dsum}_{\substack{Y<N_1 \leq X \\ N_2\leq Y}}(Y/N_1)^{(\log X)^{-1}}\sqrt{N_1N_2} \\
    & + \log Y\mathop{\dsum\dsum}_{\substack{N_1 > X \\ N_2\leq Y}}(Y/N_1)^6\sqrt{N_1N_2} 
    + \mathop{\dsum\dsum}_{\substack{N_1 \leq Y \\ Y<N_2}}(Y/N_2)^6\sqrt{N_1N_2}\\
    & +\mathop{\dsum\dsum}_{\substack{Y < N_1 \leq X \\ Y < N_2}}(Y/N_2)^6(Y/N_1)^{(\log X)^{-1}}\sqrt{N_1N_2}
    + \mathop{\dsum\dsum}_{\substack{N_1 > Y \\ N_2>X}}(Y/N_1)^6(Y/N_2)^6\sqrt{N_1N_2}
    \bigg\} \\
    & \ll \frac{a^2b}{(\log X)^{99-\epsilon}},
    \end{aligned}
\end{equation*}
upon choosing $Z = (\log X)^{80}$.

Similarly we may show that 
\[
    T^{3,n}_{Q'}(l) \ll \frac{a^2b}{2^l (\log X)^{99-\epsilon}}
\]
for $l \geq 0$,
and consequently 
\begin{equation}\label{S3nQ'aleqZ}
    S^{3,n}_{Q'}(a\leq Z) \ll \frac{X}{(\log X)^{10}}.
\end{equation}

Thus combining \eqref{S3Q'a>Z}, \eqref{S3dQ'aleqZ} and \eqref{S3nQ'aleqZ}, we conclude that 
\[
    S^{3}_{Q'} \ll X,
\]
completing the proof of Proposition 3. $\square$

\section{Estimating $\mathfrak{I}_4$}
In this section we treat $\mathfrak{I}_4$ and prove Proposition 4.

As before, we first write $\mathfrak{I}_4$ as a sum of four terms
\[
    \mathfrak{I}_4 = S^4_1 + i^{\kappa_1}\eta_fS^4_{q_1} - i^{\kappa_2}\eta_gS^4_{q_2} - i^{\kappa_1 + \kappa_2}\eta_f\eta_gS^4_{q_1q_2},
\]
where 
\begin{align*}
    S^4_{Q'} & := \psum_{(d,2Q)=1}\sum_{m}\sum_{n}\frac{\lambda_f(m)\lambda_g(n)}{\sqrt{mn}}\chi_{8d}(mnQ')J\bigg(\frac{8d}{X}\bigg) \\
    & \times \bigg\{W_1\bigg(\frac{m}{8|d|\sqrt{q_1}}\bigg) - W_1\bigg(\frac{m}{Y}\bigg)\bigg\}\cdot \bigg\{W_2\bigg(\frac{m}{8|d|\sqrt{q_1}}\bigg) - W_2\bigg(\frac{n}{Y}\bigg)\bigg\}.
\end{align*}

Using Mobius, we as before drop the coprime and square-free conditions:
\begin{align*}
    S^4_{Q'} & = \sum_{(a,2Q)=1}\mu(a)\sum_{b|Q}\mu(b)\sum_{m}\sum_{n}\frac{\lambda_f(m)\lambda_g(n)}{\sqrt{mn}}\sum_{(d,2)=1}\chi_{8bd}(mnQ')\\
    & \times J\bigg(\frac{8a^2bd}{X}\bigg) \bigg\{W_1\bigg(\frac{m}{8a^2b|d|\sqrt{q_1}}\bigg) - W_1\bigg(\frac{m}{Y}\bigg)\bigg\}\cdot \bigg\{W_2\bigg(\frac{n}{8a^2b|d|\sqrt{q_2}}\bigg) - W_2\bigg(\frac{n}{Y}\bigg)\bigg\} \\
    & = S^4_{Q'}(a \leq Z) + S^4_{Q'}(a > Z),
\end{align*}
where $S^4_{Q'}(a \leq Z)$ and $S^4_{Q'}(a > Z)$
are the contribution in $S^4_{Q'}(N_1 \leq Y, N_2)$ corresponding to terms with $a \leq Z$ and $a > Z$. As before, invoking Proposition 6.1 of \cite{Li-MR4768632} and Proposition 3.1 of \cite{Kumar.etc-MR4765788} as well as Cauchy-Schwarz, we have the following estimate for $S^4_{Q'}(a > Z)$:
\begin{equation}\label{S4Q'a>Z}
    S^4_{Q'}(a > Z) \ll \frac{X(\log X)^{1+\epsilon}}{Z^{1-\epsilon}}.
\end{equation}

It remains to treat $S^4_{Q'}(a \leq Z)$.
We split the summation over $m$ and $n$ into dyadic intervals, and re-write $S^4_{Q'}$ as
\begin{align*}
    S^4_{Q'}(a \leq Z)& = \dsum_{N_1}\dsum_{N_2}\sum_{\substack{(a,2Q)=1 \\ a \leq Z}}\mu(a)\sum_{b|Q}\mu(b)\sum_{m}\sum_{n}\frac{\lambda_f(m)\lambda_g(n)}{\sqrt{mn}}\sum_{(d,2)=1}\chi_{8bd}(mnQ')\\
    & \times J\bigg(\frac{8a^2bd}{X}\bigg) \bigg\{W_1\bigg(\frac{m}{8a^2b|d|\sqrt{q_1}}\bigg) - W_1\bigg(\frac{m}{Y}\bigg)\bigg\}\cdot \bigg\{W_2\bigg(\frac{n}{8a^2b|d|\sqrt{q_2}}\bigg) - W_2\bigg(\frac{n}{Y}\bigg)\bigg\}.
\end{align*}
We shall treat different dyadic summation ranges differently in the following. 

\subsection{$N_1 \leq Y$ or $N_2 \leq Y$}
\bigskip
 It suffices to study 
\begin{align*}
    S^4_{Q'}(a\leq Z;N_1 \leq Y, N_2) & = \dsum_{N_1\leq Y}\dsum_{N_2}\sum_{\substack{(a,2Q)=1 \\ a \leq Z}}\mu(a)\sum_{b|Q}\mu(b)\sum_{m}\sum_{n}\frac{\lambda_f(m)\lambda_g(n)}{\sqrt{mn}}\sum_{(d,2)=1}\chi_{8bd}(mnQ')\\
    & \times J\bigg(\frac{8a^2bd}{X}\bigg) \bigg\{W_1\bigg(\frac{m}{8a^2b|d|\sqrt{q_1}}\bigg) - W_1\bigg(\frac{m}{Y}\bigg)\bigg\}\cdot \bigg\{W_2\bigg(\frac{n}{8a^2b|d|\sqrt{q_2}}\bigg) - W_2\bigg(\frac{n}{Y}\bigg)\bigg\},
\end{align*}
\begin{align*}
    S^4_{Q'}(a\leq Z;N_2 \leq Y, N_1) & = \dsum_{N_1}\dsum_{N_2\leq Y}\sum_{\substack{(a,2Q)=1 \\ a \leq Z}}\mu(a)\sum_{b|Q}\mu(b)\sum_{m}\sum_{n}\frac{\lambda_f(m)\lambda_g(n)}{\sqrt{mn}}\sum_{(d,2)=1}\chi_{8bd}(mnQ')\\
    & \times J\bigg(\frac{8a^2bd}{X}\bigg) \bigg\{W_1\bigg(\frac{m}{8a^2b|d|\sqrt{q_1}}\bigg) - W_1\bigg(\frac{m}{Y}\bigg)\bigg\}\cdot \bigg\{W_2\bigg(\frac{n}{8a^2b|d|\sqrt{q_2}}\bigg) - W_2\bigg(\frac{n}{Y}\bigg)\bigg\},
\end{align*}
and 
\begin{align*}
    S^4_{Q'}(a\leq Z;N_1 \leq Y, N_2\leq Y) & = \dsum_{N_1\leq Y}\dsum_{N_2\leq Y}\sum_{\substack{(a,2Q)=1 \\ a \leq Z}}\mu(a)\sum_{b|Q}\mu(b)\sum_{m}\sum_{n}\frac{\lambda_f(m)\lambda_g(n)}{\sqrt{mn}}\sum_{(d,2)=1}\chi_{8bd}(mnQ')\\
    & \times J\bigg(\frac{8a^2bd}{X}\bigg) \bigg\{W_1\bigg(\frac{m}{8a^2b|d|\sqrt{q_1}}\bigg) - W_1\bigg(\frac{m}{Y}\bigg)\bigg\}\cdot \bigg\{W_2\bigg(\frac{n}{8a^2b|d|\sqrt{q_2}}\bigg) - W_2\bigg(\frac{n}{Y}\bigg)\bigg\}.
\end{align*}
We shall mainly focus on bounding $ S^4_{Q'}(N_1 \leq Y, N_2)$, as the estimation for $ S^4_{Q'}(N_1 \leq Y, N_2 \leq Y)$ shall follow along the way.

For the simplicity of notation, we shall write $S^4_{Q'}(a\leq Z';N_1\leq Y, N_2)$ as $T^4_{Q'}(a\leq Z)$. We firstly invoke Mellin inversion on $G(m/N_1)$ and $G(n/N_2)$ as well as the definitions for $W_1$ and $W_2$ to re-write $T^4_{Q'}(a \leq Z)$ as
\begin{equation*}
    \begin{aligned}
        T^4_{Q'}(a\leq Z) & = \dsum_{N_2}\dsum_{N_1 \leq Y}(2\pi i)^{-4}\sum_{\substack{(a,2Q)=1 \\ a \leq Z}}\mu(a)\sum_{b|Q}\mu(b)\int_{(3)}\int_{(3)}\gamma_1(w)\gamma_2(u)\\
        & \times (2\pi)^{-(w+u)}w^{-1}u^{-2}\int_{(0)}\int_{(0)}N_1^{v_1-w}N_2^{v_2-u}\tilde{G}(v_1-w)\tilde{G}(v_2-u)\\
        &\times \mathop{\sum\sum}_{(mn,2a)=1}\frac{\lambda_f(m)\lambda_g(n)}{m^{1/2+v_1}n^{1/2+v_2}}\chi_b(mnQ')\sum_{(d,2)=1}\chi_{8d}(mnQ')H_{w,u}\bigg(\frac{8a^2bd}{X}\bigg)dv_2dv_1dudw,
    \end{aligned}
\end{equation*} 
in which 
\[
    H_{w,u}(x) = J(x)\big((|x|\sqrt{q_1}X)^w-Y^w\big)\big((|x|\sqrt{q_2}X)^u-Y^u\big).
\]
Then, we apply Poisson summation formula and arrive at
\begin{equation*}
    \begin{aligned}
        T^4_{Q'}(a\leq Z) & = \frac{X}{16}\dsum_{N_2}\dsum_{N_1 \leq Y}(2\pi i)^{-4}\sum_{\substack{(a,2Q)=1 \\ a \leq Z}}\frac{\mu(a)}{a^2}\sum_{b|Q}\frac{\mu(b)}{b}\int_{(3)}\int_{(3)}\gamma_1(w)\gamma_2(u)\\
        & \times (2\pi)^{-(w+u)}w^{-1}u^{-2}\int_{(0)}\int_{(0)}N_1^{v_1-w}N_2^{v_2-u}\tilde{G}(v_1-w)\tilde{G}(v_2-u)\\
        &\times \mathop{\sum\sum}_{(mn,2a)=1}\frac{\lambda_f(m)\lambda_g(n)}{m^{1/2+v_1}n^{1/2+v_2}}\chi_b(mnQ')\sum_{k\in \mathbb{Z}}(-1)^k\frac{G_k(mnQ')}{mnQ'}\check{H}_{w,u}\bigg(\frac{kX}{16a^2bmnQ'}\bigg)dv_2dv_1dudw.
    \end{aligned}
\end{equation*}
As before, we denote the contribution from the term with $k=0$ and those with $k \neq 0$ as $T^{4,d}_{Q'}(a \leq Z)$ and $T^{4,n}_{Q'}(a\leq Z)$.
The diagonal term $T^{4,d}_{Q'}(a \leq Z)$ could be re-written as 
\begin{equation*}
    \begin{aligned}
        T^{4,d}_{Q'}(a \leq Z) & = \frac{X}{2\pi^2} (2\pi i)^{-4}
        \dsum_{N_2}\dsum_{N_1\leq Y}\int_{(3)}\int_{(3)}\gamma_1(w)\gamma_2(u)\\
        & \times (2\pi)^{-(w+u)}w^{-1}u^{-2}\check{H}_{w,u}(0)\int_{(0)}\int_{(0)}N_1^{v_1-w}N_2^{v_2-u}\tilde{G}(v_1-w)\tilde{G}(v_2-u)\\
        &\times \mathop{\sum\sum}_{\substack{(mn,2)=1\\mnQ' = \square}}\frac{\lambda_f(m)\lambda_g(n)}{m^{1/2+v_1}n^{1/2+v_2}}\prod_{p|mnQ}(1-p^{-1})\big(\prod_{p|mnQ}(1-p^{-2})^{-1}+O(Z^{-1})\big)dv_2dv_1dudw \\
        & = \frac{X}{2\pi^2} (2\pi i)^{-4}\dsum_{N_2}\dsum_{N_1\leq Y}\int_{\mathbb{R}}\int_{(3)}\int_{(3)}\gamma_1(w)\gamma_2(u)\\
        & \times (2\pi)^{-(w+u)}w^{-1}u^{-2}H_{w,u}(x)\int_{(0)}\int_{(0)}N_1^{v_1-w}N_2^{v_2-u}\tilde{G}(v_1-w)\tilde{G}(v_2-u)\\
        & \times \bigg\{
        L(1+v_1+v_2,f\otimes g)L(1+2v_1, \text{Sym}^2f)L(1+2v_2, \text{Sym}^2g)\mathcal{Z}(v_1,v_2;Q')\\
        & + O(Z^{-1})\cdot \mathop{\sum\sum}_{\substack{(mn,2)=1\\mnQ' = \square}}\frac{\lambda_f(m)\lambda_g(n)}{m^{1/2+v_1}n^{1/2+v_2}}\prod_{p|mnQ}(1-p^{-1})
        \bigg\}dv_2dv_1dudwdx \\
        & = \mathcal{T}_0 + \mathcal{T}_1 
    \end{aligned}
\end{equation*}
in which $\mathcal{T}_0$ and $\mathcal{T}_1$ are the contributions from the terms corresponding to the first and second term in the last second and third lines in the above.
Noticing that for a fixed $x \neq 0$, $H_{w,u}(x)$ is holomorphic and $H_{0,u}(x) = 0$ for any $u$ and $H_{w,0}(x) = 0$ for any $w$, we  move the integral lines $\Re(v_1), \Re(v_2) = 0$ for $v_1,v_2$-integrals to $\Re(v_1), \Re(v_2) = -\frac{1}{5}$, and then shift the integral lines $\Re(w), \Re(w) = 3$ for the $w,u$-integrals to $\Re(w), \Re(u) = -\frac{1}{10}$, passing through one simple pole at $u = 0$. By the Residue Theorem now we conclude that 
\begin{equation*}
    \mathcal{T}_0 \ll_{f,g} X^{\frac{9}{10}+\epsilon}.
\end{equation*}
On the other hand, we notice that after a change of variable for $v_1$ and $v_2$:
\begin{equation*}
    \begin{aligned}
        \mathcal{T}_1 & = X\cdot \int_{\mathbb{R}}J(x)\dsum_{N_2}\dsum_{N_1 \leq Y}\mathop{\sum\sum}_{\substack{(mn,2)=1\\mnQ' = \square}}G(m/N_1)G(n/N_2)\frac{\lambda_f(m)\lambda_g(n)}{\sqrt{mn}}\prod_{p|mnQ'}(1-p^{-1}) \\
        & \times \bigg\{W_1\bigg(\frac{m}{|x|X}\bigg) - W_1\bigg(\frac{m}{Y}\bigg)\bigg\}\cdot \bigg\{W_2\bigg(\frac{m}{|x|X}\bigg) - W_2\bigg(\frac{m}{Y}\bigg)\bigg\}\cdot O(Z^{-1})dx.
    \end{aligned}
\end{equation*}
Thus by taking absolute value on the integrand, we trivially estimate $\mathcal{T}_1$ to be $\ll_{f,g}\frac{X(\log X)^{63}}{Z}$, and hence conclude 
\begin{equation}\label{T4dQ'aleqZ}
    T^{4,d}_{Q'}(a \leq Z) \ll_{f,g} X^{\frac{9}{10}+\epsilon} + \frac{X(\log
    X)^{63}}{Z}.
\end{equation}

We then treat the off-diagonal term $T^{4,n}_{Q'}(a \leq Z)$, which is given by
\begin{equation}\label{T4offdiagonal}
    \begin{aligned}
       T^{4,n}_{Q'}(a \leq Z)& = \frac{X}{16}\dsum_{N_2}\dsum_{N_1 \leq Y}(2\pi i)^{-4}\sum_{\substack{(a,2Q)=1 \\ a \leq Z}}\frac{\mu(a)}{a^2}\sum_{b|Q}\frac{\mu(b)}{b}\int_{(3)}\int_{(3)}\gamma_1(w)\gamma_2(u)\\
        & \times (2\pi)^{-(w+u)}w^{-1}u^{-2}\int_{(0)}\int_{(0)}N_1^{v_1-w}N_2^{v_2-u}\tilde{G}(v_1-w)\tilde{G}(v_2-u)\\
        &\times \mathop{\sum\sum}_{(mn,2a)=1}\frac{\lambda_f(m)\lambda_g(n)}{m^{1/2+v_1}n^{1/2+v_2}}\chi_b(mnQ')\sum_{k\neq 0}(-1)^k\frac{G_k(mnQ')}{mnQ'}\check{H}_{w,u}\bigg(\frac{kX}{16a^2bmnQ'}\bigg)dv_2dv_1dudw \\
        & = \frac{X}{16}\sum_{\substack{(a,2Q)=1 \\ a \leq Z}}\frac{\mu(a)}{a^2}\sum_{b|Q}\frac{\mu(b)}{b}\big\{
        \sum_{l\geq 1}T^{4,n}_{Q'}(l) - T^{4,n}_{Q'}(0)\big\},
    \end{aligned}
\end{equation}
where for $l \geq 0$
\begin{equation*}
    \begin{aligned}
        T^{4,n}_{Q'}(l) &:= 
        \dsum_{N_2}\dsum_{N_1\leq Y}(2\pi i)^{-5}\int_{(\frac{1}{2})}\int^{\infty}_0J(x)x^{-s}\int_{(3)}\int_{(3)}\int_{(0)}\int_{(0)}\gamma_1(w)\gamma_2(u) \\
        & \times \frac{(xX)^w - Y^w}{w}\frac{(xX)^u - Y^u}{u^2}\tilde{G}(v_1-w)\tilde{G}(v_2-u) \\
        & \times N_1^{v_1-w}N_2^{v_2-u}\psum_{(k_1,2)=1}\sum_{2\nmid k_2 \geq 1}\mathop{\sum\sum}_{(mn,2a)=1}\frac{\lambda_f(m)\lambda_g(n)}{m^{1/2+v_1}n^{1/2+v_2}}\chi_b(mnQ')\frac{G_{2^{\delta(l)}k_1k_2^2}(mnQ')}{mnQ'}\\
        & \times \Gamma(s)\big(\cos(\pi s/2) + \text{sgn}(k_1)\sin(\pi s/2)\big)\bigg(\frac{2\pi \cdot 2^l|k_1|k_2^2X}{16a^2bmnQ'}\bigg)^{-s}dv_2dv_1dudwdxds,
        \end{aligned}
\end{equation*}
where $\delta(l) = 1$ if $(2,l)=1$ and is 0 otherwise.
Due to the structural similarity of $T^{4,n}_{Q'}(l)$'s, it suffices to estimate $T^{4,n}_{Q'}(0)$. 

We notice that upon an application of Mellin inversion on $\tilde{G}$'s, then inserting test functions $V(m/N_1)$ and $V_2(n/N_2)$, we may write $T^{4,n}_{Q'}(0)$ as 
\begin{equation*}
    \begin{aligned}
        T^{4,n}_{Q'}(0) &:= 
        \dsum_{N_2}\dsum_{N_1\leq Y}(2\pi i)^{-5}\int_{(\frac{1}{2})}\int^{\infty}_0J(x)x^{-s}\int_{(3)}\int_{(3)}\int_{(0)}\int_{(0)}\gamma_1(w)\gamma_2(u) \\
        & \times \frac{(xX)^w - Y^w}{w}\frac{(xX)^u - Y^u}{u^2}\tilde{G}(v_1-w)\tilde{G}(v_2-u) \\
        & \times N_1^{v_1-w}N_2^{v_2-u}\psum_{(k_1,2)=1}\sum_{2\nmid k_2 \geq 1}\mathop{\sum\sum}_{(mn,2a)=1}\frac{\lambda_f(m)\lambda_g(n)}{m^{1/2+v_1}n^{1/2+v_2}}\chi_b(mnQ')\frac{G_{k_1k_2^2}(mnQ')}{mnQ'}\\
        & \times \Gamma(s)\big(\cos(\pi s/2) + \text{sgn}(k_1)\sin(\pi s/2)\big)\bigg(\frac{2\pi \cdot |k_1|k_2^2X}{16a^2bmnQ'}\bigg)^{-s}dv_2dv_1dudwdxds,
        \end{aligned}
\end{equation*}
Now, we split the dyadic summation in $N_2$ into three ranges: $N_2 \leq Y$, $Y < N_2 \leq X$ and $N_2>X$, and correspondingly write $T_{Q'}^{4,n}(0)$ as
\[
    T_{Q'}^{4,n}(0) = T_{Q'}^{4,n}(N_1 \leq Y, N_2\leq Y) + T_{Q'}^{4,n}(N_1 \leq Y, Y < N_2\leq X)+ T_{Q'}^{4,n}(N_1 \leq Y, N_2>X).
\]

For $ T_{Q'}^{4,n}(N_1 \leq Y, N_2\leq Y)$, we shift the integral line $\Re(w) = \Re(u) = -6$, passing through a simple pole at $u = 0$. Using the Residue Theorem, we reach
\begin{equation*}
\begin{aligned}
    & T_{Q'}^{4,n}(N_1 \leq Y, N_2\leq Y) \ll \dsum_{N_1\leq Y}\dsum_{N_2\leq Y}\log(X/Y)\int_{-\infty}^{\infty}\int^{\infty}_{-\infty}\int^{\infty}_{-\infty}\frac{1}{(1+|-6+ it_3|)^{20}}\cdot \frac{1}{(1+|t_2|)^{20}} \\
    & \times \frac{1}{(1+|6+i(t_1-t_3)|)^{20}} \cdot \bigg\{\bigg|\psum_{2\nmid k_1}T(k_1, 1/2+it_1, 1/2+it_2,Q')\bigg| + \bigg|\psum_{2\nmid k_1}T(k_1, 1/2+it_1, 1/2+it_2,Q',1/2+it_3)\bigg| 
    \bigg\}\\
    & \ll \frac{a^2b}{X}\cdot \log(X/Y)\dsum_{N_1 \leq Y}\dsum_{N_2\leq Y}\sqrt{N_1N_2} \ll \frac{a^2bY\log(X/Y)}{X},
\end{aligned}
\end{equation*}
in which the estimation on $T(k_1, 1/2+it_1, 1/2+it_2,Q')$ and $T(k_1, 1/2+it_1, 1/2+it_2,Q',1/2+it_3)$ follow from (5.9) of \cite{zhou2025momentderivativesquadratictwists}.

For $T_{Q'}^{4,n}(N_1 \leq Y, Y < N_2\leq X)$, we shift the integral lines for $w$ and $u$ to  $\Re(w)=-6$ and $\Re(u) = (\log X)^{-1}$, and similarly as above we derive 
\begin{equation*}
    \begin{aligned}
        & T_{Q'}^{4,n}(N_1 \leq Y, N_2\leq Y) \ll \dsum_{N_1\leq Y}\dsum_{Y < N_2\leq X}\log(X/Y)^2\int_{-\infty}^{\infty}\int_{-\infty}^{\infty}\int^{\infty}_{-\infty}\int^{\infty}_{-\infty}\frac{1}{(1+|-6+ it_3|)^{20}}\frac{1}{(1+|(\log X)^{-1}+ it_4|)^{20}}\\
        & \times \frac{1}{(1+|6+i(t_1-t_3)|)^{20}} \frac{1}{(1+|-(\log X)^{-1}+i(t_2-t_4)|)^{20}}\bigg\{\bigg|\psum_{2\nmid k_1}T(k_1, 1/2+it_1, 1/2+it_2,Q')\bigg|\\
        & + \bigg|\psum_{2\nmid k_1}T(k_1, 1/2+it_1, 1/2+it_2,Q',1/2+it_3)\bigg| +  \bigg|\psum_{2\nmid k_1}T(k_1, 1/2+it_1, 1/2+it_2,Q',1/2+it_4)\bigg|\\
        &+ \bigg|\psum_{2\nmid k_1}T(k_1, 1/2+it_1, 1/2+it_2,Q',1/2+i(t_3+t_4))\bigg|
        \bigg\}\\
        & \ll \frac{a^2b}{X}\cdot \log(X/Y)^2\dsum_{N_1 \leq Y}\dsum_{Y < N_2\leq X}(Y/N_2)^{\frac{1}{\log X}}\sqrt{N_1N_2} \ll \frac{a^2b\sqrt{XY}\log(X/Y)^2}{X}.
    \end{aligned}
\end{equation*}

For $T^{4,n}_{Q'}(N_1 \leq Y, N_2 > X)$, we shift the integral lines for $w$ and $u$ to $\Re(w) = -6$ and $\Re(u) = 6$, and derive 
\begin{equation*}
    \begin{aligned}
        & T_{Q'}^{4,n}(N_1 \leq Y, N_2\leq Y) \ll \dsum_{N_1\leq Y}\dsum_{X < N_2}\int_{-\infty}^{\infty}\int_{-\infty}^{\infty}\int^{\infty}_{-\infty}\int^{\infty}_{-\infty}\frac{1}{(1+|-6+ it_3|)^{20}}\frac{1}{(1+|6+ it_4|)^{20}}\\
        & \times \frac{1}{(1+|6+i(t_1-t_3)|)^{20}} \frac{1}{(1+|-6+i(t_2-t_4)|)^{20}}\bigg\{\bigg|\psum_{2\nmid k_1}T(k_1, 1/2+it_1, 1/2+it_2,Q')\bigg|\\
        & + \bigg|\psum_{2\nmid k_1}T(k_1, 1/2+it_1, 1/2+it_2,Q',1/2+it_3)\bigg| +  \bigg|\psum_{2\nmid k_1}T(k_1, 1/2+it_1, 1/2+it_2,Q',1/2+it_4)\bigg|\\
        &+ \bigg|\psum_{2\nmid k_1}T(k_1, 1/2+it_1, 1/2+it_2,Q',1/2+i(t_3+t_4))\bigg|
        \bigg\}\\
        & \ll \frac{a^2b}{X}\dsum_{N_1 \leq Y}\dsum_{X < N_2}\sqrt{N_1N_2}\bigg(\frac{X}{N_2}\bigg)^6 \ll \frac{a^2b\sqrt{XY}}{X}.
    \end{aligned}
\end{equation*}
With these we may conclude that 
\[
    T^{4,n}_{Q'}(0) \ll \frac{a^2b}{(\log X)^{99-\epsilon}}
\]
as we have taken $Y = X/(\log X)^{200}$, and similarly (with the aid of Lemma 5.3 of \cite{zhou2025momentderivativesquadratictwists})
\begin{equation*}
    T^{4,n}_{Q'}(l) \ll \frac{a^2b}{(\log X)^{99-\epsilon}\cdot 2^l}.
\end{equation*}
Combining the above and \eqref{T4offdiagonal}, we conclude that 
\begin{equation}\label{T4nQ'aleqZ}
    T^{4,n}_{Q'}(a\leq Z) \ll \frac{XZ}{(\log X)^{99-\epsilon}}.
\end{equation}

Combining \eqref{T4dQ'aleqZ}, \eqref{T4nQ'aleqZ} and take $Z = (\log X)^{80}$ we conclude that 
\begin{equation*}
    S^4_{Q'}(a\leq Z; N_1\leq Y, N_2) = T^4_{Q'}(a\leq Z) \ll \frac{X}{(\log X)^{10}} \ll X.
\end{equation*}

The above also concludes that 
\[
    S^4_{Q'}(a\leq Z; N_1\leq Y, N_2\leq Y) \ll \frac{X}{(\log X)^{10}} \ll X.
\]
As for $S^4_{Q'}(a \leq Z; N_2\leq Y, N_1)$, we could proceed to estimate almost identically as in the case for $S^4_{Q'}(a\leq Z; N_1\leq Y,N_2)$, except that when shifting the integral lines for $w$ and $u$ in the off-diagonal, we swap the roles of $w$ and $u$, which shall cause minor differences in the expression of the integral (as now we shall always pass through a simple pole at $u=0$) which shall not affect the estimates, and we still have
\[
    S^4_{Q'}(a\leq Z; N_1\leq Y, N_2) \ll X.
\]

\subsection{$Y < N_1,N_2 \leq X$}
In this subsection, following section 7.4 of \cite{zhou2025momentderivativesquadratictwists}, we estimate the contribution from $S^{4}_{Q'}(a\leq Z; Y < N_1, N_2 \leq X)$. Applying Cauchy-Schwarz to the $d$-summation, as well as noticing the positivity, we have 
\begin{equation*}
    \begin{aligned}
     S^{4}_{Q'}(a\leq Z; Y < N_1, N_2 \leq X) & \leq \dsum_{Y< N_1\leq X}\bigg(\psum_{(d,2Q)=1}J(8d/X)\bigg|\sum_{m\in \mathbb{Z}}\frac{\lambda_f(m)\chi_{8d}(m)}{\sqrt{m}}G(m/N_1)\\
     & \times \big(W_1\bigg(\frac{m}{8d\sqrt{q_1}}\bigg) - W_1\bigg(\frac{m}{Y}\bigg) \big)\bigg|^2\bigg)^{\frac{1}{2}} \\
     & \times \dsum_{Y< N_2\leq X}\bigg(\psum_{(d,2Q)=1}J(8d/X)\bigg|\sum_{n\in \mathbb{Z}}\frac{\lambda_g(n)\chi_{8d}(n)}{\sqrt{n}}G(n/N_2)\\
     & \times \big(W_2\bigg(\frac{n}{8d\sqrt{q_2}}\bigg) - W_2\bigg(\frac{n}{Y}\bigg) \big)\bigg|^2\bigg)^{\frac{1}{2}}.
    \end{aligned} 
\end{equation*}

We estimate the first factor of the above as follows:
\begin{lemma}\label{f-factor in Y<N1leqX range}
    With the notations above, we have
    \[
        \dsum_{Y< N_1\leq X}\bigg(\psum_{(d,2Q)=1}J(8d/X)\bigg|\sum_{m\in \mathbb{Z}}\frac{\lambda_f(m)\chi_{8d}(m)}{\sqrt{m}}G(m/N_1)
        \big(W_1\bigg(\frac{m}{8d\sqrt{q_1}}\bigg) - W_1\bigg(\frac{m}{Y}\bigg) \big)\bigg|^2\bigg)^{\frac{1}{2}} \ll X^{\frac{1}{2}}(\log \log X)^2.
    \]
\end{lemma}
\begin{proof}
    Denoting the $d$-summation under the square root in the above as $\mathcal{R}$, we firstly insert the test function $V$ into the above as 
    \begin{equation*}
        \begin{aligned}
            \mathcal{R} & = \psum_{(d,2Q)=1}J(8d/X)\bigg|\sum_{m}\frac{\lambda_f(m)\chi_{8d}(m)}{\sqrt{m}} G(m/N_1)V(m/N_1) \\
            & \big(W_1\bigg(\frac{m}{8d\sqrt{q_1}}\bigg) - W_1\bigg(\frac{m}{Y}\bigg) \big)\bigg|^2 \\
            & = \psum_{(d,2Q)=1}J(8d/X)\bigg|(2\pi i)^{-2}\int_{(3)}\int_{(0)}\tilde {G}(s)N_1^s(2\pi)^{-w}\gamma_1(w)w^{-1}\big\{(8d\sqrt{q_1})^w - Y^w
            \big\}\\
            & \times \sum_{m}\frac{\lambda_f(m)\chi_{8d}(m)}{m^{w+s+\frac{1}{2}}}V(m/N_1)dsdw
            \bigg|^2.
        \end{aligned}
    \end{equation*}
    In the above, we make a change of variable on $w$, and then shift the $w$-integral line to $\Re(w) = (\log X)^{-1}$, and then invoke Cauchy-Schwarz to the double integral to arrive at 
    \begin{equation*}
        \begin{aligned}
            \mathcal{R} & \ll \bigg(\frac{Y}{N_1}\bigg)^{2(\log X)^{-1}}\times   \psum_{(d,2Q)=1}J(8d/X)\int_{((\log X)^{-1})}\int_{(0)} |\gamma_1(w)||\tilde{G}(s-w)|\bigg|\frac{(8d\sqrt{q_1}/Y)^w - 1}{\sqrt{w}}\bigg|^2dsdw \\
            & \times \int_{((\log X)^{-1})}\int_{(0)} |\gamma_1(w)||\tilde{G}(s-w)||w|^{-1}\bigg|\sum_{m}\frac{\lambda_f(m)\chi_{8d}(m)}{m^{\frac{1}{2}+s}}V(m/N_1)\bigg|^2dsdw.
        \end{aligned}
    \end{equation*}
    We treat the first double integral above first. Invoking the rapid decay of $\tilde{G}(\cdot)$ and $\gamma_1(\cdot)$, uniformly in $d$ we have
    \begin{equation*}
        \begin{aligned}
            & \int_{((\log X)^{-1})}\int_{(0)} |\gamma_1(w)||\tilde{G}(s-w)|\bigg|\frac{(8d\sqrt{q_1}/Y)^w - 1}{\sqrt{w}}\bigg|^2dsdw \\
            & \ll \int^{\infty}_{-\infty}\int^{\infty}_{-\infty}\frac{|(\log X)^{-1}+it_1|}{(1+ |(\log X)^{-1} + it_1|)^{20}\cdot (1+ |-(\log X)^{-1} + i(t_2-t_1)|)^{20}} \\
            & \times \bigg|\frac{(8d\sqrt{q_1}/Y)^{(
            \log X
            )^{-1} + it_1}-1}{(\log X)^{-1}+it_1}\bigg|^2dt_2dt_1 \\
            & \ll \sup_{t \in \mathbb{R}}\bigg\{\bigg|\frac{(8d\sqrt{q_1}/Y)^{(
            \log X
            )^{-1} + it}-1}{(\log X)^{-1}+it}\bigg|^2\bigg\}\times \int^{\infty}_{-\infty}\frac{1}{(1+|t_1|)^{19}}dt_1 \ll (\log \log X)^2,
        \end{aligned}
    \end{equation*}
    in which we have used the fact that 
    \[
    \frac{(8d\sqrt{q_1}/Y)^{(
            \log X
            )^{-1} + it}-1}{(\log X)^{-1}+it}\ll \log \log X
    \]
    for all $t\in \mathbb{R}$ as $\frac{X}{16}\leq d \leq \frac{X}{4}$.

    As for the rest 
    \[
        \psum_{(d,2Q)=1}J(8d/X)\int_{((\log X)^{-1})}\int_{(0)} |\gamma_1(w)||\tilde{G}(s-w)||w|^{-1}\bigg|\sum_{m}\frac{\lambda_f(m)\chi_{8d}(m)}{m^{\frac{1}{2}+s}}V(m/N_1)\bigg|^2dsdw,
    \]
    we apply Lemma \ref{Li-biliear} to conclude that the above is bounded by $X$. Thus 
    \[
        \dsum_{Y<N_1 \leq X}\mathcal{R}^{\frac{1}{2}} \ll \dsum_{Y < N_1 \leq X}(Y/N_1)^{(\log X)^{-1}}X^{\frac{1}{2}}(\log \log X) \ll X^{\frac{1}{2}}(\log \log X)^2,
    \]
    finishing the proof.
\end{proof}

Now, combining Lemma \ref{f-factor in Y<N1leqX range} and Lemma 7.1 of \cite{zhou2025momentderivativesquadratictwists}, we have 
\begin{equation*}
    S^4_{Q'}(a\leq Z; Y < N_1, N_2 \leq X) \ll X(\log \log X)^5.
\end{equation*}

\subsection{Neither do $N_1 \leq Y$ nor $N_2\leq Y$}
In this section we treat $S^4_{Q'}(a\leq Z; Y< N_1\leq X, X<N_2)$, $S^4_{Q'}(a\leq Z; Y< N_2 \leq X, X < N_1)$ and $S^4_{Q'}(a\leq Z;X < N_1, X<N_2)$. We only estimate $S^4_{Q'}(a\leq Z; Y< N_1\leq X, X<N_2)$ in full detail as the other two could be treated similarly. 

For $S^4_{Q'}(a\leq Z; Y< N_1\leq X, X<N_2)$, we apply Cauchy-Schwarz and positivity to get 
\begin{equation*}
    \begin{aligned}
    S^{4}_{Q'}(a\leq Z; Y < N_1, N_2 \leq X) & \leq \dsum_{Y< N_1\leq X}\bigg(\psum_{(d,2Q)=1}J(8d/X)\bigg|\sum_{m\in \mathbb{Z}}\frac{\lambda_f(m)\chi_{8d}(m)}{\sqrt{m}}G(m/N_1)\\
     & \times \big(W_1\bigg(\frac{m}{8d\sqrt{q_1}}\bigg) - W_1\bigg(\frac{m}{Y}\bigg) \big)\bigg|^2\bigg)^{\frac{1}{2}} \\
     & \times \dsum_{X< N_2}\bigg(\psum_{(d,2Q)=1}J(8d/X)\bigg|\sum_{n\in \mathbb{Z}}\frac{\lambda_g(n)\chi_{8d}(n)}{\sqrt{n}}G(n/N_2)\\
     & \times \big(W_2\bigg(\frac{n}{8d\sqrt{q_2}}\bigg) - W_2\bigg(\frac{n}{Y}\bigg) \big)\bigg|^2\bigg)^{\frac{1}{2}}.
    \end{aligned}
\end{equation*}
For the summand
\begin{equation*}
    \psum_{(d,2Q)=1}J(8d/X)\bigg|\sum_{n\in \mathbb{Z}}\frac{\lambda_g(n)\chi_{8d}(n)}{\sqrt{n}}G(n/N_2) \times \big(W_2\bigg(\frac{n}{8d\sqrt{q_2}}\bigg) - W_2\bigg(\frac{n}{Y}\bigg) \big)\bigg|^2
\end{equation*}
of the $d$-summation in the second factor above,
we proceed as in the previous subsection to rewrite it as 
\begin{equation*}
    \begin{aligned}
        & \psum_{(d,2Q)=1}J(8d/X)\bigg|(2\pi i)^{-2}\int_{(3)}\int_{(0)}\tilde {G}(s)N_2^s(2\pi)^{-u
        }\gamma_2(u)u^{-2}\big\{(8d\sqrt{q_2})^u - Y^u
            \big\}\\
        & \times \sum_{n}\frac{\lambda_g(n)\chi_{8d}(n)}{n^{u+s+\frac{1}{2}}}V(n/N_2)dsdu
        \bigg|^2 \\
        & = \psum_{(d,2Q)=1}J(8d/X)\bigg|(2\pi i)^{-2}\int_{(3)}\int_{(0)}\tilde{G}(s-u)N_2^{s-u}(2\pi)^{-u}\gamma_2(u)u^{-2}\big\{(8d\sqrt{q_2})^u - Y^u\big\}\\
        & \times \sum_{n}\frac{\lambda_g(n)\chi_{8d}(n)}{n^{s+\frac{1}{2}}}V(n/N_2)dsdu\bigg|^2 \\
        & \ll \bigg(\frac{X}{N_2}\bigg)^6\bigg(\int_{-\infty}^{\infty}\int^{\infty}_{-\infty}
        \frac{1}{|3+it_1|\cdot (1+|3+it_1|)^{20}\cdot (1 + |-3+i(t_2-t_1)|)^{20}}
        dt_2dt_1\bigg)\\
        & \times \bigg(\int_{-\infty}^{\infty}\int^{\infty}_{-\infty}
        \frac{1}{|3+it_1|\cdot (1+|3+it_1|)^{20}\cdot (1 + |-3+i(t_2-t_1)|)^{20}}\psum_{(d,2Q)=1}J(8d/X)\\
        & \bigg|\sum_{n}\frac{\lambda_g(n)\chi_{8d}(n)}{n^{s+\frac{1}{2}}}V(n/N_2)\bigg|^2
        dt_2dt_1\bigg)
    \end{aligned}
\end{equation*}
By Lemma 3.5 of \cite{zhou2025momentderivativesquadratictwists}, the above is bounded by $(X/N)^6X$. Now, by Lemma \ref{f-factor in Y<N1leqX range}, we conclude 
\[
    S^4_{Q'}(a\leq Z; Y < N_1 \leq X, X < N_2) \ll X(\log \log X)^2.
\]
Similarly, we have 
\[
    S^4_{Q'}(a\leq Z; Y < N_2 \leq X, X < N_1) \ll X(\log \log X)^3
\]
by Lemma 7.1 of \cite{zhou2025momentderivativesquadratictwists},
and 
\[
    S^4_{Q'}(a\leq Z;  N_1,N_2>X) \ll X.
\]

Therefore we by the estimations in sub-sections 8.2, 8.3 and 8.4, may deduce that with $Z = (\log X)^{80}$:
\begin{equation}\label{S4nQ'aleqZ}
    S^{4}_{Q'}(a \leq Z) \ll X(\log \log X)^5.
\end{equation}
Now, from \eqref{S4Q'a>Z} and \eqref{S4nQ'aleqZ}, we conclude that 
\[
    S^4_{Q'} \ll X(\log \log X)^5,
\]
finishing the proof of Proposition \ref{BfB'g}. $\square$

\bigskip 

\bibliographystyle{plain}
\bibliography{Bib.bib}

\end{document}